\documentclass[11pt]{article}
\usepackage{IEEEtrantools}
\usepackage{latexsym}
\usepackage{amssymb}
\usepackage{amsmath,amsfonts,amsthm}
\usepackage{mathrsfs}
\usepackage[all]{xy}
\usepackage{epic}
\usepackage{eepic}
\usepackage{enumerate}
\usepackage{multirow}
\usepackage[breaklinks]{hyperref}
\usepackage{tikz}

\usepackage{mathrsfs}
\usetikzlibrary{matrix,arrows,decorations.pathmorphing}
\usepackage{color}
\usepackage[breaklinks]{hyperref}
\usepackage{tikz-cd}

\usetikzlibrary{arrows,decorations.markings}
\usepackage{caption}
\usepackage[mathcal]{euscript}
\usetikzlibrary{fit,matrix}
\usepackage{authblk}
\usepackage{dsfont}

\newcommand{\bbbt}{\mathbb{T}}
\newcommand{\scrt}{\mathscr{T}}

\newcommand{\be}{\begin{equation}}
	\newcommand{\ee}{\end{equation}}
\newcommand{\bea}{\begin{eqnarray}}
	\newcommand{\eea}{\end{eqnarray}}
\newcommand{\bean}{\begin{eqnarray*}}
	\newcommand{\eean}{\end{eqnarray*}}
\newcommand{\brray}{\begin{array}}
	\newcommand{\erray}{\end{array}}
\newcommand{\biearray}{\begin{IEEEarray}{rCl}}
	\newcommand{\eiearray}{\end{IEEEarray}}

\newtheorem{dfn}{Definition}[section]
\newtheorem{thm}[dfn]{Theorem}
\newtheorem{lmma}[dfn]{Lemma}
\newtheorem{ppsn}[dfn]{Proposition}
\newtheorem{crlre}[dfn]{Corollary}
\newtheorem{xmpl}[dfn]{Example}
\newtheorem{rmrk}[dfn]{Remark}

\newcommand{\bdfn}{\begin{dfn}\rm}
	\newcommand{\bthm}{\begin{thm}}
		\newcommand{\blmma}{\begin{lmma}}
			\newcommand{\bppsn}{\begin{ppsn}}
				\newcommand{\bcrlre}{\begin{crlre}}
					\newcommand{\bxmpl}{\begin{xmpl}}
						\newcommand{\brmrk}{\begin{rmrk}\rm}
							
							\newcommand{\edfn}{\end{dfn}}
						\newcommand{\ethm}{\end{thm}}
					\newcommand{\elmma}{\end{lmma}}
				\newcommand{\eppsn}{\end{ppsn}}
			\newcommand{\ecrlre}{\end{crlre}}
		\newcommand{\exmpl}{\end{xmpl}}
	\newcommand{\ermrk}{\end{rmrk}}

\newcommand{\bbc}{\mathbb{C}}
\newcommand{\bbz}{\mathbb{Z}}
\newcommand{\bbn}{\mathbb{N}}
\newcommand{\bbr}{\mathbb{R}}
\newcommand{\bbq}{\mathbb{Q}}

\newcommand{\cla}{\mathcal{A}}

\newcommand{\clh}{\mathcal{H}}

\newcommand{\clk}{\mathcal{K}}
\newcommand{\cll}{\mathcal{L}}

\newcommand{\prf}{\noindent{\it Proof\/}: }

\newcommand{\one}{{1\!\!1}}

\newcommand{\id}{\mbox{id}}

\def \qed { \mbox{}\hfill
	$\Box$\vspace{1ex}}

\newcommand\restr[2]{{% we make the whole thing an ordinary symbol
		\left.\kern-\nulldelimiterspace % automatically resize the bar with \right
		#1 % the function
		\littletaller % pretend it's a little taller at normal size
		\right|_{#2} % this is the delimiter
}}

\newcommand{\littletaller}{\mathchoice{\vphantom{\big|}}{}{}{}}

\begin{document}

	%%%%%%%%%%%%%%%%%%%%%%%%%%%%%%%%%
	\author{{\sc Shreema Subhash Bhatt, Surajit Biswas, Bipul Saurabh}}
\date{}
	\title{On the classification of $C^*$-algebras of twisted isometries with finite dimensional wandering spaces}
	\maketitle
	%%%%%%%%%%%%%%%%%%%%%%%%%%%%%%%%%%
	%%%%%  ABSTRACT
	%%%%%%%%%%%%%%%%%%%%%%%%%%%%%%%%%%
	
	\begin{abstract}
		Let \( m, n \in \mathbb{N}_0 \), and let \( X \) be a closed subset of \( \mathbb{T}^{\binom{m+n}{2}} \). We define \( C^{m,n}_X \) to be the universal \( C^* \)-algebra among those generated by \( m \) unitaries and \( n \) isometries satisfying doubly twisted commutation relations with respect to a twist \( \mathcal{U} = \{U_{ij}\}_{1 \leq i < j \leq m+n} \) of commuting unitaries having joint spectrum \( X \).
		We provide a complete list of the irreducible representations of \( C^{m,n}_X \) up to unitary equivalence and, under a denseness assumption on \( X \), explicitly construct a faithful representation of \( C^{m,n}_X \). Under the same assumption, we also give a necessary and sufficient condition on a fixed tuple \( \mathcal{U} \) of commuting unitaries with joint spectrum \( X \) for the existence of a universal tuple of \( \mathcal{U} \)-doubly twisted isometries.
		For \( X = \mathbb{T}^{\binom{m+n}{2}} \), we compute the \( K \)-groups of \( C^{m,n}_X \). We further classify the \( C^* \)-algebras generated by a pair of doubly twisted isometries with a fixed parameter \( \theta \in \mathbb{R} \setminus \mathbb{Q} \), whose wandering spaces are finite-dimensional. Finally, for a fixed unitary \( U \), we classify all the \( C^* \)-algebras generated by a pair of \( U \)-doubly twisted isometries with finite-dimensional wandering spaces.
		
	\end{abstract}
	
		\textbf{AMS Subject Classification No.:} 46L85, 46L80,  46L05.


	\tableofcontents
	
	\section{Introduction}
	
	The problem of classifying objects within a category is a fundamental one, and 
    \( C^* \)-algebras form no exception. 
    The Elliott classification program, which began in the 1990s (see \cite{Ell-1995aa}), aims to classify 
    \( C^* \)-algebras up to isomorphism using computable invariants such as \( K \)-theory and traces. 
    Remarkable progress has been made for large classes of nuclear, separable, simple \( C^* \)-algebras, 
    particularly those satisfying the Universal Coefficient Theorem (UCT) and exhibiting regularity 
    properties such as finite nuclear dimension or \( \mathcal{Z} \)-stability. 
    However, for non-simple \( C^* \)-algebras, the classification problem remains significantly more intricate.
    
    This article focuses on $C^*$-algebras generated by isometries and unitaries satisfying twisted commutation relations. The structure of such \( C^* \)-algebras have been extensively studied from various perspectives (see, e.g., \cite{Web-2013aa}, \cite{DejPin-2016aa}, \cite{RakSarSur-2022aa}, \cite{RakSarSur-2022ab}). 
    A fundamental example is the noncommutative torus \( \mathcal{A}_\Theta^n \), the universal \( C^* \)-algebra generated by unitaries \( u_1, \ldots, u_n \) satisfying \( u_i u_j = e^{2\pi i \theta_{ij}} u_j u_i \), for a real skew-symmetric matrix \( \Theta = (\theta_{ij}) \) \cite{Rie-1990aa}. 
    Proskurin \cite{Pro-2000aa} extended this framework by studying \( C^* \)-algebras \( \mathcal{A}_{0,\Theta} \) generated by isometries \( v_1, \ldots, v_n \) satisfying the twisted relation \( v_i^* v_j = e^{2\pi i \theta_{ij}} v_j v_i^* \), and further representation-theoretic results for these algebras were obtained by Kabluchko \cite{Kab-2001aa}. 
    For a fixed \( \binom{m+n}{2} \)-tuple of commuting unitaries \( \mathcal{U} \), Jaydeb et al.\ \cite{RakSarSur-2022aa} considered tuples of \( \mathcal{U} \)-doubly twisted isometries and proved that any such tuple admits a Wold–von Neumann decomposition. 
    Weber \cite{Web-2013aa} examined \( C^* \)-algebras generated by two isometries twisted either by a tensor-type or a free-type twist, analyzing their ideal structures and computing their \( K \)-theory. 
    Recently, Bhatt and Saurabh \cite{BhaSau-2023aa} explored a higher-dimensional generalization involving \( m \) unitaries and \( n \) isometries, namely the \( C^* \)-algebras \( C^{m,n}_{\Theta} \), and analyzed their representation theory and \( K \)-stability.

    In this article, our main object of study is the universal \( C^* \)-algebra, denoted by \( C^{m,n}_X \), in the category of \( C^* \)-algebras generated by \( \mathcal{U} \)-doubly twisted \( m \) unitaries and \( n \) isometries, where \( \mathcal{U} \) is any \( \binom{m+n}{2} \)-tuple of commuting unitaries whose joint spectrum \( \sigma(\mathcal{U}) \) is contained in \( X \). 
     To state the problem clearly, without delving into technical jargon yet still capturing the essence, we assume for now that \( m = 0 \). A natural question arises: given a \( \binom{n}{2} \)-tuple of commuting unitaries \( \mathcal{U} \), with \( \sigma(\mathcal{U}) = X \), acting on a Hilbert space \( \mathcal{H} \), does there exist an \( n \)-tuple of \( \mathcal{U} \)-doubly twisted isometries such that the \( C^* \)-algebra generated by this tuple is canonically isomorphic to \( C^{0,n}_X \)? Do the multiplicities arising in the spectral decomposition of \( \mathcal{U} \) also play a role in ensuring the existence of such a tuple?
     Here we show that this is indeed the case. Under the assumption that the set of non-degenerate points in \( X \) is dense, we prove that such a tuple exists if and only if the spectral multiplicities of points in \( X \) are infinite.
     
   The next question we address in this paper pertains to the classification of \( C^* \)-algebras generated by a pair of doubly twisted isometries. Let \( (T_1, T_2) \) be a pair of such isometries satisfying \( T_1^* T_2 = e^{2\pi i \theta} T_2 T_1^* \) for some \( \theta \in \mathbb{R} \setminus \mathbb{Q} \). Under the assumption that all wandering subspaces arising in the von Neumann--Wold-type orthogonal decomposition of \( (T_1, T_2) \) are finite-dimensional, we classify the resulting \( C^* \)-algebra. We compute the \( K \)-groups of these \( C^* \)-algebras and prove that the \( K \)-groups serve as complete invariants for this class of \( C^* \)-algebras. 
   Next, for a fixed unitary \( U \), we consider a pair of \( U \)-doubly twisted isometries with finite-dimensional wandering spaces and show that, in such cases, the spectrum of \( U \) is a finite set. Using this fact, we classify all the \( C^* \)-algebras generated by such pairs. This analysis relies crucially on the orthogonal decomposition techniques developed in \cite{RakSarSur-2022aa} and the tools of \( C^* \)-extension theory developed in \cite{Lin-2009aa}.

In the next section, we recall some basic definitions and preliminaries needed for the rest of the paper. Section~$2$ deals with the representation theory of \( C^{m,n}_X \); under a mild condition on \( X \), we describe an explicit faithful representation of \( C^{m,n}_X \), and then settle the first question raised above. Using the \( C(X) \)-algebra structure of \( C^{m,n}_X \), we show that its center coincides with the \( C^* \)-subalgebra generated by the unitary tuple \( \mathbb{U} \). Section~$4$ computes the \( K \)-groups of \( C^{m,n}_X \) when \( X = \mathbb{T}^{\binom{m+n}{2}} \). The final section is devoted to the classification of \( C^* \)-algebras generated by a pair of doubly twisted isometries with finite-dimensional wandering spaces.

	Throughout, all Hilbert spaces and $C^*$-algebras are assumed to be separable and defined over the complex field $\mathbb{C}$. We denote $\mathbb{N}_0 := \mathbb{N} \cup \{0\}$. The standard orthonormal bases of $\ell^2\left(\mathbb{N}_0\right)$ and $\ell^2\left(\mathbb{Z}\right)$ are denoted by $\{e_n\}_{n \in \mathbb{N}_0}$ and $\{e_n\}_{n \in \mathbb{Z}}$, respectively. The right shift $e_n \mapsto e_{n+1}$ is denoted by $S^*$, and the number operator by $N$. The Toeplitz algebra generated by $S^*$ is denoted $\mathcal{T}$, and $\mathcal{K}$ denotes the algebra of compact operators. For a $C^*$-algebra $A$ and elements $a_1, \dots, a_n \in A$, we write $\langle a_1, \dots, a_n \rangle$ for the closed two-sided ideal they generate. The unit circle in $\mathbb{C}$ is denoted by $\mathbb{T}$, and $\sigma(\mathcal{U})$ denotes the joint spectrum of a commuting tuple of unitaries $\mathcal{U} = \langle U_{ij} \rangle_{1 \leq i < j \leq m+n}$.  
	Let $\Omega_k$ denote the set of skew-symmetric real $k \times k$ matrices, and let $\Omega_k^*$ be the subset of $\Omega_k$ consisting of all nondegenerate elements.

    \section{Preliminaries}
    In this section, we introduce some definitions and notations which will be used throughout the paper. 
  \bdfn \label{isometry}
  Set $C^{0,0}=\bbc$, $C^{1,0}=C(\bbbt)$ and $C^{0,1}=\scrt$. For $m+n>1$ and  $\Theta=\{\theta_{ij}:1 \leq i <j \leq m+n\}$,  define $C_{\Theta}^{m,n}$ to be the universal $C^*$-algebra generated by 
  $s_1, s_2, \cdots s_{m+n}$ satisfying the following relations;
  \begin{IEEEeqnarray*}{rCll}
  	s_i^*s_j&=&e^{-2\pi \mathrm{i} \theta_{ij} }s_js_i^*,\,\,& \mbox{ if } 1\leq i<j \leq m+n;\\
  	s_i^*s_i&=& 1,\,\, & \mbox{ if } 1\leq i\leq m+n;\\
  	s_is_i^*&=&1, \,\, & \mbox{ if }  1\leq i \leq m.
  \end{IEEEeqnarray*}
  \edfn
  
  \bdfn \label{main definition} Let $m, n \in \bbn_0$ with $m+n\geq 2$ and let $X$ be a closed subset of  $\bbbt^{{m+n \choose 2}}$. 	We define $C_{X}^{m,n}$ to be the universal $C^*$-algebra generated by 
  $s_1, s_2, \cdots s_{m+n}$ and $\mathbb{U}:=\{u_{ij}:1\leq i <j \leq m+n\}$ 
  such that   
  \begin{IEEEeqnarray}{rCll} \label{R2}
  	s_i^*s_j&=&u_{ij}^*s_js_i^*\,\,& \mbox{ if } 1\leq i<j \leq m+n,\\
  	s_i^*s_i&=& 1\,\, & \mbox{ if } 1\leq i\leq m+n,\\
  	s_is_i^*&=&1 \,\, & \mbox{ if }  1\leq i \leq m,\\
  	s_i u_{pq}&=&u_{pq}s_i\,\, & \mbox{ if }  1\leq i \leq m+n, 1\leq p<q \leq m+n,\\
  	u_{ij}u_{ij}^*&=& 	u_{ij}^*u_{ij}=1 \, \, & \mbox{ if }  1\leq i <j \leq m+n,\\
  	u_{pq}u_{ij}&=& u_{ij}u_{pq}\,\, & \mbox{ if }  1\leq i <j \leq m+n, 1\leq p<q \leq m+n,
  \end{IEEEeqnarray} 
  and  the joint spectrum 	$\sigma(\mathbb{U})\subset X$.
  \edfn 	
  \brmrk
  \begin{enumerate} [(i)]
  	\item If $X$ be a singleton set consisting the tuple $(e^{2\pi \mathrm{i} \theta_{ij}} : 1 \leq i<j\leq m+n )$, then $C_{X}^{m,n}=C_{\Theta}^{m,n}$. 
  	\item If all the unitaries $u_{ij}$ are such that $\sigma(u_{ij})=\bbbt^{{m+n \choose 2}}$, then we denote $C_{X}^{m,n}$ by $C_{\mathbb{U}}^{m,n}$ where $\mathbb{U}=\{u_{ij}:1\leq i<j\leq m+n\}$.
    \item For notational consistency, we define $C_{\mathbb{U}}^{1,0}$ to be the universal $C^*$-algebra generated by a unitary element $s_1$, so that $C_{\mathbb{U}}^{1,0} \cong C(\mathbb{T})$. Similarly, we define $C_{\mathbb{U}}^{0,1}$ to be the universal $C^*$-algebra generated by an isometry $s_1$, i.e., the Toeplitz algebra.
    \end{enumerate}
    \ermrk 
  \bppsn \label{joint spectrum}  In $C_{X}^{m,n}$, the joint spectrum   of $\mathbb{U}=\{u_{ij}:1\leq i <j \leq m+n\}$ is  $X$. 
  As a consequence, the $C^*$-algebra generated by $\{u_{ij}:1\leq i <j \leq m+n\}$ is isomorphic to $C(X)$. 
  \eppsn 
  \prf
  From the definition of $C_{X}^{m,n}$, it follows that 	$\sigma(\mathbb{U})\subset X$.  Take  $x \in X$. Any   representation $\psi$   of $C_{x}^{m,n}$  induces  a representation $\Psi$ of  $C_{X}^{m,n}$ by mapping $u_{ij}$ to $x_{ij}$ and $s_k$ to $\psi(s_k)$ as $x_{ij}$'s and $\psi(s_k)$'s satisfy the defining relations of $C_{X}^{m,n}$. This proves that $x \in \sigma(\mathbb{U})$. 
  \qed \\
  The above argument also provides  concrete realizations of  $C_{X}^{m,n}$ via representations of $C_{x}^{m,n}$ for $x \in X$;  thereby establishing its existence. 
  
  \subsection{Wold-decomposition} 
  \bdfn \label{definition of U}
  Let $\mathcal{U}=\{U_{ij}\}_{1\leq  i<j\leq m+n}$  be a $\binom{m+n}{2}$-tuple of commuting unitaries acting on a Hilbert space $\clh$. An tuple $\mathcal{T}=(T_1,T_2,\cdots , T_{m+n})$ of isometries on $\clh$ is called $\mathcal{U}$-doubly twisted $(m,n)$-isometries if 
  \begin{IEEEeqnarray}{rCll}
  	T_i U_{st}&=&U_{st}T_i, & \mbox{ for } 1\leq i \leq m+n, \, 1 \leq s < t\leq m+n, \label{center}\\
  	T_i^*T_j&=&U_{ij}^*T_jT_i^*, \quad & \mbox{ for } 1 \leq i < j \leq m+n. \label{doubly twisted relation}\\
  	T_iT_i^* &=& 1, & \mbox{ for } 1\leq i \leq m.
  \end{IEEEeqnarray}
  \edfn
  In such a case, we write $\mathcal{T} \leadsto C^{m,n}_X$, where $\sigma(\mathcal{U})=X$.  If $m=0$ the we call $\mathcal{T}$ a tuple of $\mathcal{U}$-doubly twisted isometries.  For any tuple  $\mathcal{T}=(T_1,T_2,\cdots , T_{m+n})$ of $\mathcal{U}$-doubly twisted $(m,n)$-isometries, the von-Neumann-Wold decomposition holds (see \cite{RakSarSur-2022aa}). More precisely, one gets a decomposition  $\clh= \oplus_{A\subset \Sigma_{m,n}}\clh_A$, where $\clh_A$'s are reducing subspaces and  
  \begin{IEEEeqnarray*}{rCl}
  	T_i &=& \oplus_{A\subset\{m+1,\cdots, m+n\}}T_i^A, \, \mbox{ where } T_i^A=\restr{T_i}{\clh_A}  \mbox{ for  } A\subset\{m+1,\cdots, m+n\} \mbox{ and } 1\leq i \leq m+n,
  \end{IEEEeqnarray*}
  and $T_i^A$ is an isometry if $i \in A$, and unitary if $i \notin A$. Furthermore, $U_{ij}$'s reduce $\clh_A$. We denote $\restr{U_{ij}}{\clh_A}$ by $U_{ij}^A$. The joint spectrum of $\{U_{ij}^A: 1\leq  i<j\leq n\}$ is denoted by $\sigma^A(\mathcal{U})$. Moreover,
  from the relation (\ref{doubly twisted relation}), one has 
  $$ U_{ij}^A= (T_i^A)^*T_j^A(T_j^A)^*T_i^A, \mbox{ for } i \neq j.$$
  Define $W_A=\cap_{i \in A} \ker T_i^*$ if $A$ is nonempty, and $\clh$ if $A=\phi$. We call these subspaces as the wandering subspaces of the tuple $\mathcal{T}$.
   For any  $E=\{i_{1},\ldots,i_{l}\}\in I_{m+n} $ and  $\alpha=(\alpha_{i_{1}},\ldots,\alpha_{i_{l}})\in\mathbb{N}^{|A|}$, define  $$T_{E}^{\alpha}=T_{i_{1}}^{\alpha_{i_{1}}}\ldots T_{i_{l}}^{\alpha_{i_{l}}}.$$ 
For $A \in \{m+1,\cdots , m+n\}$, one can define 
$$ B_{A}=\cap_{\alpha \in\mathbb{N}^{|A^c|}}\,\, T_{ A^c}^{\alpha }\mathcal{W}_{A}.$$
Then we have 
   $$ \clh_A=\oplus_{\gamma \in \mathbb{N}^{|A|}} \, \, T_{A}^{\gamma} B_A.$$
  Note that  $W_{A}$ is a reducing subspace for $T_{j}$ for all $j\in A^{c}$, in particular, for the unitary operators $T_{1},\ldots,T_{m}$, and therefore,  for all $i\in I_m$, we have $T_{i}W_{A}=W_{A}$ and hence $T_{i}^{\alpha_{i}}W_{A}=W_{A}$ for all $\alpha_{i}\in\mathbb{N}$. Further, since  $W_{A}$ is a reducing subspace for $U_{ij}$ for all $1\leq i\neq j\leq m+n$, hence $U_{ij}W_A=W_A$, and $T_iT_j=U_{ij}T_jT_i$, it follows that 
  $$B_A=\cap_{\alpha \in\mathbb{N}^{|A^c\setminus I_m|}}\,\, T_{ A^c\setminus I_m}^{\alpha}\mathcal{W}_{A} \quad \mbox{ and } \quad 
   \mathcal{H}_{A}=\oplus_{\gamma \in\mathbb{N}^{|A|}}T_{A}^{\gamma}(B_A). $$
   In other words,  the unitary elements $T_{1},\ldots,T_{m}$ do not appear in the expression of $B_{A}$. If we take,  $A=\{m+1,\ldots,m+n\}$, then one has  $B_{A}=W_{A}$.
  
  \subsection{$C(X)$-algebra structure}
  Since $u_{ij}$'s are in the center of $C_{X}^{m,n}$, we  get a homomorphism
  \[
  \beta: C(X) \rightarrow Z(C_{X}^{m,n}), \quad f(x)\mapsto f(\mathbb{U}).
  \]
  This map gives  $C_{X}^{m,n}$ a $C(X)$-algebra structure.  For $x \in X$, define $\mathcal{I}_{x}$ to be the ideal of $C_{X}^{m,n}$ generated by $\{\beta(f-f(x)): f \in C(X)\}=\{f(\mathbb{U}):f \in C(X)\}$. 
  Let  $\pi_{x}: C_{X}^{m,n} \rightarrow C_{X}^{m,n}/\mathcal{I}_{x}$  to be the quotient map. 
  Write $\pi_{x}(a)$ as $a_{x}$ for $a \in C_{X}^{m,n}$. The following theorem establishes  $C_{X}^{m,n}$  a continuous $C(X)$-algebra. 
  \bppsn \label{continuous algebra}
  Let $m,n\in \bbn_0$ with $m+n\geq 2$. Then the  $C^*$-algebra 
  $C_{X}^{m,n}$ is a continuous $C(X)$-algebra.
  \eppsn
  \prf  
  It follows from  [\cite{BhaSau-2023aa}, Theorem $5.4$] that the  $C^*$-algebra generated by any tuple of 
  $\mathcal{U}$-doubly twisted $(m,n)$-isometries is a continuous $C(\sigma(\mathcal{U}))$-algebra. If one takes a faithful representation $\zeta$ of $C_{X}^{m,n}$ and $\mathcal{U}=\pi(\mathbb{U})$ then $\mathcal{S}=(\zeta(s_1),\cdots, \zeta(s_{m+n})$ is a $\mathcal{U}$-doubly twisted $(m,n)$-isometries. From Proposition \ref{joint spectrum}, we get $\sigma(\mathcal{U})=X$. This proves that $C_{X}^{m,n}$ is a continuous $C(X)$-algebra.
  \qed \\
  In the following proposition, we show that the fibre  of $C_{X}^{m,n}$ at the point $x$ is $C_{x}^{m,n}$. To do so, we require a result, which we will prove in the next section.
  \bppsn 
  Let $x \in X$. Then there exists an isomorphism   $\Psi_x: C_{X}^{m,n} /\mathcal{I}_{x} \rightarrow  C_{x}^{m,n}$ sending $[s_i]$ to $s_i$ and $[u_{ij}]$ to $x_{ij}\one$.
  \eppsn 
  \prf  Consider the map 
  $$ \Phi_x: C_{X}^{m,n}  \rightarrow  C_{x}^{m,n}; \quad s_i \mapsto s_i,  \quad u_{ij} \mapsto  x_{ij}\one.$$
  That $\Phi_x$ is a surjective homomorphism follows from the  defining relations (\ref{main definition}, \ref{isometry}) of $C_{X}^{m,n}$ and $C_{x}^{m,n}$. Note that $\ker \Phi_x \supset \mathcal{I}_x$. This induces a surjective homomorphism $\Psi_x: C_{X}^{m,n} /\mathcal{I}_{x} \rightarrow  C_{x}^{m,n}$. To show injectivity of $\Psi_x$, it is enough to show that each irreducible representation of  $C_{X}^{m,n}$ that vanishes on $\mathcal{I}_{x}$ factors through  $\Phi_x$, which can easily be checked using Theorem \ref{representations1}.  \qed 
  \brmrk \label{convention}
 Due to the above isomorphism,  we identify $\pi_x$ with $\Psi_x \circ \pi_x$ and view $a_x$ as an element of $C_{x}^{m,n}$. 
  \ermrk 
  The following propositios will come in handy in establishing classification results in this paper, as it enables us to disregard certain superfluous components of the 
  $C^*$-algebra of twisted isometries.
  \bppsn \label{isomdirectsum}
  Let $A$ be a $C^*$-algebra
  and let  $\phi_{\alpha}:A\rightarrow B_{\alpha}$ be a $*$-homomorphisms for all $\alpha \in \bigwedge$.  Define $C$ to be the $C^*$-subalgebra of $\oplus_{\alpha \in \bigwedge}B_{\alpha}$  generated by $\{\oplus_{\alpha\in \bigwedge}\phi_{\alpha}(a): a \in A\}$. If for some $\alpha_0\in \bigwedge$, the homomorphism $\phi_{\alpha_0}$ is injective then $A$ is isomorphic to $C$. 
  \eppsn 
  \prf Consider the $*$-homomorphism 
  $$ \oplus_{\alpha\in \bigwedge}\phi_{\alpha}: A \rightarrow \oplus_{\alpha \in \bigwedge}B_{\alpha}; \quad a \mapsto \oplus_{\alpha\in \bigwedge}\phi_{\alpha}(a).$$ 
  Since  $\ker \phi_{\alpha_0}=\{0\}$, we get
  $$\ker \oplus_{\alpha\in \bigwedge}\phi_{\alpha}=\cap_{\alpha \in \bigwedge} \ker \phi_{\alpha}=\{0\}.$$
  Therefore the map $\oplus_{\alpha\in \bigwedge}\phi_{\alpha}$ gives an isomorphism from $A$ onto $C$. 
  \qed  \\
  We mention two more results of this type and omit the proof, as it is straightforward to verify.
   \bppsn \label{isomdirectsum1}
  Let $A$ be a $C^*$-algebra and let  $\phi_{\alpha}:A\rightarrow B_{\alpha}$ be a $*$-homomorphisms for all $\alpha \in \bigwedge$.    Suppose that  there exists $\bigwedge^{\prime}\subset \bigwedge$ and  $\alpha_0 \in \bigwedge \setminus \bigwedge^{\prime}$ such that  $B_{\alpha}=B_{\alpha_0}$ and $\phi_{\alpha}=\phi_{\alpha_0}$ for all $\alpha \in \bigwedge^{\prime}$. Then the $C^*$-subalgebra of $\oplus_{\alpha \in \bigwedge}B_{\alpha}$  generated by $\{\oplus_{\alpha\in \bigwedge}\phi_{\alpha}(a): a \in A\}$ is isomorphic to the $C^*$-subalgebra of $\oplus_{\alpha \in \bigwedge \setminus \bigwedge^{\prime}}\, \, B_{\alpha}$  generated by $\{\oplus_{\alpha\in \bigwedge \setminus \bigwedge^{\prime}} \, \, \phi_{\alpha}(a): a \in A\}$ .
  \eppsn 
  
   \bppsn \label{isomdirectsum2}
  Let $A$ be a $C^*$-algebra and let  $\phi_{\alpha}:A\rightarrow B_{\alpha}$ be a $*$-homomorphisms for all $\alpha \in \bigwedge$.    Suppose that  there exists $\bigwedge^{\prime}\subset \bigwedge$ and  $\alpha_0 \in \bigwedge \setminus \bigwedge^{\prime}$, and a homomorphism $\phi_{\alpha}^0: B_{\alpha_0} \rightarrow B_{\alpha}$ for all $\alpha \in \bigwedge^{\prime}$ such that $\phi_{\alpha}= .\phi_{\alpha}^0 \circ \phi_{\alpha_0}$. Then the $C^*$-subalgebra of $\oplus_{\alpha \in \bigwedge}B_{\alpha}$  generated by $\{\oplus_{\alpha\in \bigwedge}\phi_{\alpha}(a): a \in A\}$ is isomorphic to the $C^*$-subalgebra of $\oplus_{\alpha \in \bigwedge \setminus \bigwedge^{\prime}}\, \, B_{\alpha}$  generated by $\{\oplus_{\alpha\in \bigwedge \setminus \bigwedge^{\prime}} \, \, \phi_{\alpha}(a): a \in A\}$ .
  \eppsn

\section{Faithful representation of $C^{m,n}_X$}

In this section, we  will  describe all irreducible representations of $C_{X}^{m,n}$, and using that, we obtain a faithful representation of $X$.	Fix $m,n \in \bbn_0$ such that $m+n\geq 2$,  
define   $$\Sigma_{m,n}=\{I\subset \{1,2,\cdots ,m+n\}: \{1,2,\cdots m\} \subset I\} \mbox{ and } \Theta_I=\{\theta_{ij}:1 \leq i<j \leq m+n, i,j\in I\}.$$
Fix $I=\{i_1<i_2<\cdots  < i_r\} \in \Sigma_{m+n}$. Let $I^c=\{j_1<j_2<\cdots < j_s\}$. 	Let $\rho:\cla_{\Theta_I}^{|I|} \rightarrow \mathcal{L}(K)$ be a unital representation. Define a map $\pi_{(I,\rho)}^{\Theta}$ of the $C^*$-algebra $C_{\Theta}^{m,n}$ on the Hilbert space $\clh^{I}:=K\otimes \ell^2(\bbn_0)^{\otimes (m+n-|I|)}$ as follows:
\begin{IEEEeqnarray*}{rCll}
	\pi_{(I,\rho)}^{\Theta} :C_{\Theta}^{m,n} &\rightarrow &\mathcal{L}(\clh^{I})\\
	s_{j_l} &\mapsto& 
	1 \otimes 1^{\otimes^{l-1}}\otimes S^* \otimes e^{2\pi \mathrm{i} \theta_{j_lj_{l+1}} N}\otimes  \cdots \otimes e^{2\pi \mathrm{i} \theta_{j_lj_{s}} N}, &\quad \mbox{ for } 1 \leq l \leq s, \\
	s_{i_l} &\mapsto& 
	\pi(s_{i_l}) \otimes \lambda_{i_l,j_1}\otimes  \lambda_{i_l,j_2} \otimes  \cdots \otimes  \lambda_{i_l,j_s}&\quad \mbox{ for } 1 \leq l \leq r,
\end{IEEEeqnarray*}
where $\lambda_{i_l,j_k}$ is  $e^{-2\pi \mathrm{i} \theta_{i_l,j_k} N} $ if $i_l > j_k$ and $e^{2\pi \mathrm{i} \theta_{i_l,j_k} N}$  if $i_l < j_k$. 

\bthm \rm(\cite{BhaSau-2023aa})   \label{representations} 
The set $\{\pi_{(I,\rho)}^{\Theta}: I \subset \Sigma_{m,n}, \rho \in \widehat{\cla_{\Theta_I}^{|I|}}\}$ gives all irreducible representations of $C_{\Theta}^{m,n}$ upto unitarily equivalence. 
\ethm

By \cite[Corollary~7.6]{RakSarSur-2022aa}, the space $\widehat{C_X^{0,n}}$, where $X=\mathbb{T}^{\binom{n}{2}}$ is characterized in terms of the irreducible representations of $2^n$ noncommutative tori $\mathbb{T}_A$, where $A \subseteq \{1, 2, \dots, n\}$. In the following proposition, we characterize $\widehat{C_X^{m,n}}$ in terms of $\widehat{C_x^{m,n}}$. For any $x = \langle x_{ij} \rangle_{1 \leq i < j \leq m+n} \in X$, we identify the quotient algebra $C_X^{m,n} / \mathcal{I}_x$ with $C_x^{m,n}$ (see Proposition \ref{continuous algebra}), and denote the canonical $^*$-homomorphism by $\pi_x : C_X^{m,n} \to C_x^{m,n}$, given by
$$
\pi_x(u_{ij}) = x_{ij} I, \quad \pi_x(s_i) = s_i \quad \text{for all } i, j.
$$

      \begin{ppsn}\label{irred}
        	Let $m,n \in \mathbb{N}_0$ with $m+n \geq 2$. The map
        	\[
        	\bigsqcup_{x \in X} \widehat{C_x^{m,n}} \to \widehat{C_X^{m,n}}, \quad \left[\Phi_x\right] \mapsto \left[\Phi_x \circ \pi_x\right],
        	\]
        	where $\Phi_x \in \widehat{C_x^{m,n}}$ and $x\in X$, is a bijection.
        \end{ppsn}
        
        \begin{proof}
        	To prove surjectivity, let $\Psi \in \widehat{C_X^{m,n}}$. Since $u_{ij} \in \mathcal{Z}\left(C_X^{m,n}\right)$ and $u_{ij}^* u_{ij} = u_{ij} u_{ij}^* = 1$ for each $1 \leq i < j \leq m+n$, it follows that $\Psi(u_{ij}) = x_{ij}I$ for some $x=\left\langle x_{ij}\right\rangle_{1\leq i<j\leq m+n}\in X$. By the universality of $C_x^{m,n}$, there exists a unique $^*$-homomorphism $\Phi_x: C_x^{m,n} \to \Psi\left(C_X^{m,n}\right)$ such that for each $1 \leq i \leq m+n$, 
        	\[
        	\Phi_x(s_i) = \Psi(s_i), \quad \text{i.e.,} \quad \Phi_x \circ \pi_x(s_i) = \Psi(s_i),
        	\]
        	and for each $1 \leq i < j \leq m+n$,
        	\[
        	\Phi_x\left(x_{ij}I\right) = x_{ij}, \quad \text{i.e.,} \quad \Phi_x \circ \pi_x(u_{ij}) = \Psi(u_{ij}).
        	\]
        	Thus, $\left[\Phi_x\right]$ is a preimage of $\left[\Psi\right]$, proving surjectivity.
        	
        	To prove injectivity, let $\Phi_x^{(1)} \circ \pi_x: C_X^{m,n} \to \mathcal{L}(\mathcal{H})$ and $\Phi_y^{(2)} \circ \pi_y: C_X^{m,n} \to \mathcal{L}\left(\mathcal{H}'\right)$ be two unitarily equivalent irreducible representations. Then there exists a unitary operator $U: \mathcal{H} \to \mathcal{H}'$ such that for all $a \in C_X^{m,n}$,
        	\[
        	U \Phi_x^{(1)} \circ \pi_x(a) U^{-1} = \Phi_y^{(2)} \circ \pi_y(a).
        	\]
        	In particular, for each $1 \leq i < j \leq m+n$,
        	\[
        	U \Phi_x^{(1)} \circ \pi_x(u_{ij}) U^{-1} = \Phi_y^{(2)} \circ \pi_y(u_{ij}),
        	\]
        	which implies $x_{ij} = y_{ij}$ and hence $x = y$. Moreover, for each $1 \leq i \leq m+n$,
        	\[
        	U \Phi_x^{(1)} \circ \pi_x(s_i) U^{-1} = \Phi_x^{(2)} \circ \pi_x(s_i),
        	\]
        	i.e., $U \Phi_x^{(1)}(s_i) U^{-1} = \Phi_x^{(2)}(s_i)$. Therefore, $\Phi_x^{(1)}$ and $\Phi_x^{(2)}$ are unitarily equivalent, proving injectivity.
        \end{proof}

\bthm \label{representations1}
Let $m,n \in \bbn_0$ with $m+n\geq 2$ and let $X$ be a closed  subset of $ \bbbt^{m+n\choose 2}$. Then the set $\{\pi_{(I,\rho)}^{\Theta}: I \in \Sigma_{m,n}, \rho \in \widehat{\cla_{\Theta_I}^{|I|}}, \Theta \in X \subset \bbbt^{m+n\choose 2} \}$ gives all irreducible representations of $C_{X}^{m,n}$ upto  unitary equivalence. 
\ethm 
\prf  The claim follows from Theorem \ref{representations} and Proposition \ref{irred}.
\qed

\bppsn \label{intersection}
Let $m,n \in \bbn_0$ with $m+n\geq 2$. Let $x \in X$ and $\mathcal{I}_x$ be the closed ideal of $C^{m,n}_X$ generated by $\{u_{ij}-x_{ij}\one:1 \leq i < j \leq m+n\}$. Then one has 
$$ \bigcap_{x\in X}\, \,  \mathcal{I}_x=\{0\}.$$
\eppsn
\prf If $ a \in \bigcap_{x\in X}\, \,  \mathcal{I}_x$, then it follows from Theorem \ref{representations1} that  $\pi(a)=0$ for any irreducible representation $\pi$  of $C^{m,n}_X$. This proves that $a=0$.
\qed 
\bppsn \label{intersection1}  Let $D$ be dense in $X$.  Then one has 
$$ \bigcap_{x\in D}\, \,  \mathcal{I}_x=\{0\}.$$ 
\eppsn 
\prf Let $a \in  \bigcap_{x\in D}\, \,  \mathcal{I}_x$. Then $a_x=0$ for $x \in D$.  Take an arbitrary point $y \in X$. Choose a sequence $\{x_n\}_{n=1}^{\infty} \subset D$ such that $x_n \rightarrow y$ as $ n \rightarrow \infty$. Since $C^{m,n}_X$ is a continuous $C(X)$-algebra, we have 
$$ \|a_y\|=\lim_{n \rightarrow \infty} \|a_{x_n}\|=0.$$
The claim now follows from Proposition \ref{intersection}. 
\qed 

\bppsn\label{injective representation} 
Let $D$ be dense in $X$. Let $\beta: C^{m,n}_X \rightarrow C$ be a $*$-homomorphism such that  the homomorphism $\pi_x: C^{m,n}_X \rightarrow C^{m,n}_x$ factors through $\beta$ for all $x \in D$. Then $\beta$ is injective. 
\eppsn 
\prf Let $a \in \ker \beta$. Then $a_x =\pi_x(a)=0$ for $x \in D$ as $\pi_x$ factors through $\beta$. This implies that $a=0$, thanks to Proposition \ref{intersection1}.
\qed 
\blmma \label{Surjective to bijective}
Suppose that  $\pi:C_{x}^{0,n}\rightarrow	\mathcal{L}\left(\mathcal{H}\right)$ is a representation of the $C^*$-algebra $C_x^{0,n}$. If $\displaystyle\pi\left(\prod_{i=1}^n\left(I-s_i s_i^*\right)\right)\neq 0$, then $\pi$ is faithful.
\elmma

\begin{proof}
	Let $S_i = \pi(s_i)$ for each $1 \leq i \leq n$. By the von Neumann–Wold decomposition (see \cite{RakSarSur-2022aa}), the Hilbert space $\mathcal{H}$ decomposes as
	\[
	\mathcal{H} = \bigoplus_{A \subset I_n} \mathcal{H}_A,
	\]
	where $I_n = \{1,2,\ldots,n\}$, where in particular,
	\[
	\mathcal{H}_{I_n} = \bigoplus_{m_i \in \mathbb{N}_0,\; i \in I_n} S_1^{m_1} \cdots S_n^{m_n} W_{I_n},
	\quad \text{with} \quad
	W_{I_n} := \bigcap_{i \in I_n} \ker S_i^*.
	\]
	
	Let $p = \prod_{i=1}^n (I - s_i s_i^*)$. Then
	\[
	\pi(p) = \prod_{i=1}^n \mathrm{proj}_{\ker S_i^*}
	= \mathrm{proj}_{\bigcap_{i=1}^n \ker S_i^*}
	= \mathrm{proj}_{W_{I_n}}.
	\]
	Since $\pi(p) \neq 0$, we have $W_{I_n} \neq \{0\}$. Choose a unit vector $w \in W_{I_n}$. Consider the subspace
	\[
	\mathcal{H}_w := \bigoplus_{m_i \in \mathbb{N}_0,\; i \in I_n} S_1^{m_1} \cdots S_n^{m_n} (\mathbb{C}w),
	\]
	which is invariant under the action of $C_x^{0,n}$. Consider the isomorphism $U_w : \mathcal{H}_w \to \ell^2(\mathbb{N}_0)^{\otimes n}$ specified by
	$$U_w\left(S_1^{m_1} \cdots S_n^{m_n} w\right) = e_{m_1} \otimes \cdots \otimes e_{m_n},\,m_i\in\mathbb{N}_0,\,i\in I_n.$$
	
	To prove that $\pi$ is faithful, it suffices to show that $\pi|_{\mathcal{H}_w}$ is faithful. Define a representation $\psi: C_x^{0,n}\to\mathcal{L}\left(\ell^2\left(\mathbb{N}_0\right)^{\otimes n}\right)$ by $\psi(a) := U_w \pi|_{\mathcal{H}_w}(a) U_w^{-1}, \quad a \in C_x^{0,n}$. Then $\psi$ acts on $\ell^2(\mathbb{N}_0)^{\otimes n}$ via $\psi(s_1) = S^* \otimes I^{\otimes (n-1)}$, and for each $2 \leq i \leq n$,
	\[
	\psi(s_i) = x_{1i}^{N} \otimes \cdots \otimes x_{(i-1)i}^{N} \otimes S^* \otimes I^{\otimes (n - i)}.
	\]
	
	By \cite[Remark 2.9]{BhaSau-2023aa}, this representation $\psi$ is faithful. Therefore, $\pi|_{\mathcal{H}_w}$ is faithful, and hence $\pi$ is injective.
\end{proof}

\bcrlre \label{faithful homomorphism} 
Suppose that  $\pi:C_{x}^{0,n}\rightarrow	C$ is a $*$-homomorphism. If $\displaystyle\pi\left(\prod_{i=1}^n\left(I-s_i s_i^*\right)\right)\neq 0$, then $\pi$ is injective.
\ecrlre
\prf One can view the $C^*$-algebra $C$ as a $C^*$-subalgebra of $\cll(\clh)$ for some Hilbert space $\clh$. The claim now follows immediately from Lemma \ref{Surjective to bijective}.
\qed 
\blmma \label{Surjective to bijective m,n}
Suppose that $\pi:C_{x}^{m,n}\rightarrow	\mathcal{L}\left(\mathcal{H}\right)$ be a representation of the $C^*$-algebra $C_x^{m,n}$. Let $\mathcal{P}=\displaystyle\prod_{i=m+1}^{m+n}\left(I-s_is_i^*\right)$, and $D=C^*(\{\mathcal{P}s_1, \cdots, \mathcal{P}s_m\})$.  If $\displaystyle\pi(\mathcal{P})\neq 0$ and $\pi_{|D}$ is faithful, then $\pi$ is faithful.
\elmma
\prf The proof follows precisely along the lines of the proof of Lemma \ref{Surjective to bijective}, with an application of Theorem $2.7$ in \cite{BhaSau-2023aa}.
\qed 
\bcrlre \label{faithful homomorphism m,n}
Suppose that $\pi:C_{x}^{m,n}\rightarrow	C$ be a $*$-homomorphism. Let  $\mathcal{P}=\displaystyle\prod_{i=m+1}^{m+n}\left(I-s_is_i^*\right)$, and $D=C^*(\{\mathcal{P}s_1, \cdots, \mathcal{P}s_m\})$. If $\displaystyle\pi(\mathcal{P})\neq 0$ and $\pi_{|D}$ is injective, then $\pi$ is injective.
\ecrlre

      \begin{thm}\label{F.rep 0,n}
        	Let $n\in\mathbb{N}$ with $n\geq 2$. Let $\mathcal{U}=\left\langle U_{ij}\right\rangle_{1\leq i<j\leq n}$ be a tuple of commuting unitaries on a Hilbert space $\mathbf{K}$ such that the joint spectrum $\sigma(\mathcal{U}) = X$. For $1\leq i\leq n$, define the unitary operator $F_i$ as follows:
        	\begin{itemize}
        		\item If $i=1$, define $F_i: = I_K$,
        		\item If $2\leq i\leq n$, define $F_i : \mathbf{K}\otimes \ell^2\left(\mathbb{N}_0\right)^{\otimes (i-1)}\to\mathbf{K}\otimes\ell^2\left(\mathbb{N}_0\right)^{\otimes (i-1)}$ by 
        		$$F_i\left(k\otimes e_{m_1}\otimes\cdots\otimes e_{m_{i-1}}\right) = \left(\prod_{s=1}^{i-1} U_{si}^{-m_s}\right)k\otimes e_{m_1}\otimes\cdots\otimes e_{m_{i-1}},$$
        		where $k\in \mathbf{K}$, and $m_s\in\mathbb{N}_0$ for $1\leq s\leq i-1$.
        	\end{itemize}
        	Then, there exists a faithful $^*$-representation $\pi : C_X^{0,n}\to\mathcal{L}\left(\mathbf{K}\otimes\ell^2\left(\mathbb{N}_0\right)^{\otimes n}\right)$ such that:
        	\begin{enumerate}[(i)]
                \item For each $1\leq i<j\leq n$, $\pi\left(u_{ij}\right) = U_{ij}\otimes I$.
        		\item For each $1\leq i\leq n$, $\pi\left(s_i\right) = F_i\otimes S^*\otimes I^{\otimes (n-i)}$.
            \end{enumerate}
        \end{thm}
        
       \begin{proof}
       	\begin{figure}[h]
       		\centering
        		\[
        		\begin{tikzcd}[
    column sep=large,
    row sep=large,
    execute at end picture={
        % First \circlearrowright at (x1,y1)
        \node at (0.0, 0.9) {\scalebox{3.0}{$\circlearrowright$}};  
        % Second \circlearrowright at (x2,y2)
        \node at (0.5, -0.7) {{\scalebox{2.2}{$\circlearrowright$}}};  
    }
]
        			C^{0,n}_X 
       			\arrow[r, "\pi"] 
      			\arrow[d, "\pi_x"'] 
        			\arrow[ddr, swap, bend right=20, "\Phi"] 
        			& C 
        			\arrow[d, "q"] 
        			\\
        			C^{0,n}_x 
        			\arrow[r, "\iota"] 
        			\arrow[dr, swap, bend right=10, "\Phi_x"] 
        			& C / \overline{\mathcal{I}_x} 
        			\arrow[d, "\lambda_x"] 
        			\\
        			& \mathcal{L}\left(\mathcal{H}\right)        			\end{tikzcd}
        		\]
        		\caption{Diagram corresponding to Theorem \ref{F.rep 0,n}}
       		\label{fig:commutative_diagram}
        	\end{figure}
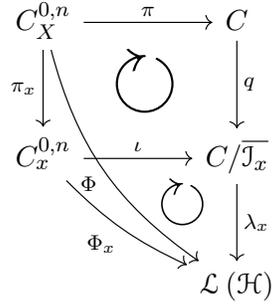
        	
        	Let $\mathcal{H}_0 = \mathbf{K}\otimes\ell^2\left(\mathbb{N}_0\right)^{\otimes n}$. By evaluating on the elements $k\otimes e_{m_1}\otimes\cdots\otimes e_{m_n}$, $k\in\mathbf{K}$, and $m_l\in\mathbb{N}_0$ for $l=1,2,\ldots,n$, we can show that the operators 
       	$$\overline{u_{ij}}:= U_{ij}\otimes I, \text{ and } \overline{s_l} := F_l\otimes S^*\otimes I^{\otimes (n-l)}, \text{ for } 1\leq i,j\leq n,\, 1\leq l\leq n$$
       	satisfy the generating relations of $C_X^{0,n}$. Hence, there exists a unique $^*$-homomorphism $\pi : C_X^{0,n}\to\mathcal{L}\left(\mathcal{H}_0\right)$ such that
        	$$\pi\left(u_{ij}\right) = \overline{u_{ij}}, \text{ and } \pi\left(s_l\right) = \overline{s_l} \text{ for } 1\leq i<j\leq n \text{ and } 1\leq l\leq n.$$
        	
        	To prove the injectivity of $\pi$, it suffices to show that every irreducible representation $\Phi : C_X^{0,n}\to\mathcal{L}\left(\mathcal{H}\right)$ factors through $\pi$, i.e., for any such $\Phi$, there exists a $^*$-representation $\lambda: \pi\left[C_X^{0,n}\right]\to\mathcal{L}\left(\mathcal{H}\right)$ satisfying $\Phi = \lambda\circ\pi$, ensuring $\ker\pi =\{0\}$.
        	
        	Let $C = \pi\left[C_X^{0,n}\right]$. Since we have for each $1\leq i<j\leq n$ that $\Phi\left(u_{ij}\right) = x_{ij}I$ for some $x_{ij}\in\mathbb{T}$, it follows that $\Phi\left(u_{ij}-x_{ij}I\right) = 0$, implying $x:= \left\langle x_{ij}\right\rangle_{1\leq i<j\leq n}\in \sigma(\mathbf{u})$, where $\mathbb{U}:=\left\langle u_{ij}\right\rangle_{1\leq i<j\leq n}$. Denote $\overline{\mathbb{U}}:= \left\langle\overline{u_{ij}}\right\rangle_{1\leq i<j\leq n}$. Consider the closed proper two-sided ideal $\mathcal{I}_x = \displaystyle\sum_{1\leq i<j\leq n} C_X^{0,n}\left(u_{ij} - x_{ij} I\right)$ of $C_X^{0,n}$. Then
        	$$\overline{\mathcal{I}_x} := \pi\left[\mathcal{I}_x\right] = \sum_{1\leq i<j\leq n} A\left(\overline{u_{ij}} - x_{ij}I\right)$$
        	is a closed proper two-sided ideal of $C$, where the properness follows from the fact that $\sigma\left(\overline{\mathbb{U}}\right) = X$.
        	
        	Let $q: C\to C/{\overline{\mathcal{I}_x}}$ denote the natural projection map. Then 
        	$$q\left(\overline{u_{ij}}\right) = x_{ij} I,\, q\left(\overline{s_l}\right), \text{ for } 1\leq i<j\leq n,\, 1\leq l\leq n$$
        	satisfy the generating relations of $C_x^{0,n}$. Hence, there exists a unique $^*$-homomorphism $\iota : C_x^{0,n}\to C/{\overline{\mathcal{I}_x}}$ such that for each $1\leq l\leq n$, $\iota\left(s_l\right) = q\left(\overline{s_l}\right)$. By evaluating on the generators, it is easy to show that $\iota\circ\pi_x = q\circ\pi$ (see Figure \ref{fig:commutative_diagram}), which in particular shows that $\iota$ is surjective.
        	
        	We now claim that $\iota\left(\left(I-s_1 s_1^*\right)\cdots\left(I-s_n s_n^*\right)\right)\neq 0$. By Lemma \ref{Surjective to bijective}, it then follows that $\iota$ is a $^*$-isomorphism. Suppose, for contradiction, that $\iota\left(\left(I-s_1s_1^*\right)\cdots\left(I-s_ns_n^*\right)\right) = 0$. Then
        	$$q\left(\left(I-\overline{s_1}\;\overline{s_1}^*\right)\cdots\left(I-\overline{s_n}\;\overline{s_n}^*\right)\right) = \iota\circ\pi_x\left(\left(I-s_1 s_1^*\right)\cdots\left(I-s_n s_n^*\right)\right) =0,$$
        	which implies that $\left(I-\overline{s_1}\;\overline{s_1}^*\right)\cdots\left(I-\overline{s_n}\;\overline{s_n}^*\right)\in \overline{\mathcal{I}_x}$.
        	
        	Since, for each $1\leq l\leq n$, $F_l$ is a unitary operator, we have 
        	$$\left(I-\overline{s_1}\;\overline{s_1}^*\right)\cdots\left(I-\overline{s_n}\;\overline{s_n}^*\right) = I\otimes\left(I-S^*S\right)^{\otimes n} = I\otimes p^{\otimes n} \in \overline{\mathcal{I}_x}.$$
        	
        	For each $1\leq i<j\leq n$, pick $c_{ij}\in C$ such that
        	$$\sum_{1\leq i<j\leq n} c_{ij}\left(\overline{u_{ij}} - x_{ij}I\right) = \sum_{1\leq i<j\leq n} \left(\overline{u_{ij}} - x_{ij}I\right) c_{ij} = I\otimes p^{\otimes n}.$$
        	
        	Consider the operators $E:\mathbf{K}\to \mathcal{H}_0$, and $Q: \mathcal{H}_0 \to \mathbf{K}$ defined by
        	$$E(k) = k\otimes e_0^{\otimes n},\text{ and } Q\left(k\otimes e_{m_1}\otimes\cdots\otimes e_{m_n}\right) = \delta_{m_1}\cdots\delta_{m_n}\cdot k, \text{ respectively,}$$
        	for $k\in\mathbf{K}$, $m_l\in\mathbb{N}_0$, $1\leq l\leq n$. Evaluating on the elements of the form $k\otimes e_{m_1}\otimes\cdots\otimes e_{m_n}$, $k\in\mathbf{K}$, $m_l\in\mathbb{N}_0$, $1\leq l\leq n$, it follows that
        	$$\sum_{1\leq i<j\leq n} \left(Qc_{ij}E\right)\left(U_{ij} - x_{ij}I\right) = \sum_{1\leq i<j\leq n} \left(U_{ij} - x_{ij} I\right)\left(Qc_{ij}E\right) = I,$$
        	i.e., $x\notin\sigma\left(\mathcal{U}\right)$, a contradiction. Hence, the claim follows, and $\iota$ becomes a $^*$-isomorphism. Therefore, there exists a $^*$-homomorphism $\lambda_x : C/{\overline{\mathcal{I}_x}}\to\mathcal{L}\left(\mathcal{H}\right)$ such that $\lambda_x\circ\iota = \Phi_x$, and then the $^*$-homomorphism $\lambda = \lambda_x\circ q$ satisfies
        	$$\lambda\circ\pi = \lambda_x\left(q\circ\pi\right) = \lambda_x\left(\iota\circ\pi_x\right) = \Phi_x\circ\pi_x =\Phi.$$
        \end{proof}
        
        \begin{thm}\label{F.rep m,n}
         Let $m,n\in\mathbb{N}_0$ with $m+n\geq 2$. Let $X$ be a closed subset of $\mathbb{T}^{\binom{m+n}{2}}$ such that $X\cap\Omega_{m+n}^*$ is dense in $X$. Let $\mathcal{U}=\left\langle U_{ij}\right\rangle_{1\leq i<j\leq n}$ be a tuple of commuting unitaries on a Hilbert space $\mathbf{K}$ such that the joint spectrum $\sigma(\mathcal{U}) = X$.
        
       For $1\leq i\leq m+n$, define the unitary operator $F_i$ as follows:
        \begin{itemize}
        \item If $i=1$, set $F_1:= I_{\mathbf{K}}$.
        \item If $2\leq i\leq m+1$, define $F_i : \mathbf{K}\otimes\ell^2\left(\mathbb{Z}\right)^{\otimes (i-1)} \to \mathbf{K}\otimes\ell^2\left(\mathbb{Z}\right)^{\otimes (i-1)}$ by
        $$F_i\left(k\otimes e_{n_1}\otimes \cdots\otimes e_{n_{i-1}}\right) = \left(\prod_{r=1}^{i-1} U_{ri}^{-n_r}\right)k\otimes e_{n_1}\otimes\cdots\otimes e_{n_{i-1}}.$$
        \item If $m+2\leq i\leq m+n$, define
        $$F_i:\mathbf{K}\otimes\ell^2\left(\mathbb{Z}\right)^{\otimes m}\otimes\ell^2\left(\mathbb{N}_0\right)^{\otimes (i-m-1)} \to \mathbf{K}\otimes\ell^2\left(\mathbb{Z}\right)^{\otimes m}\otimes\ell^2\left(\mathbb{N}_0\right)^{\otimes (i-m-1)}$$ by
        \begin{align*}
        &F_i\left(k\otimes e_{n_1}\otimes\cdots\otimes e_{n_m}\otimes e_{l_1}\otimes\cdots\otimes e_{l_{i-m-1}}\right)\\
        &:=\left(\prod_{r=1}^m U_{ri}^{-n_r}\right)\left(\prod_{s=m+1}^{i-1} U_{si}^{-l_{s-m}}\right)k\otimes e_{n_1}\otimes\cdots\otimes e_{n_m}\otimes e_{l_1}\otimes\cdots\otimes e_{l_{i-m-1}},
        \end{align*}
        \end{itemize}
        where $k\in\mathbf{K}$, $n_r\in\mathbb{Z}$ for $1\leq r\leq m$, and $l_t\in\mathbb{N}_0$ for $1\leq t\leq n-1$.
        
        Then there exists a faithful $^*$-representation $\pi : C_X^{m,n}\to\mathcal{L}\left(\mathbf{K}\otimes\ell^2\left(\mathbb{Z}\right)^{\otimes m}\otimes\ell^2\left(\mathbb{N}_0\right)^{\otimes n}\right)$ such that:
        \begin{enumerate}[(i)]
        \item For each $1\leq i<j\leq m+n$, $\pi\left(u_{ij}\right) = U_{ij}\otimes I$.
        \item For each $1\leq i\leq m+n$, $\pi\left(s_i\right) = F_i\otimes S^*\otimes I$.
        \end{enumerate}
        \end{thm}
        
        \begin{proof}
        Let $\mathcal{H}_0 = \mathbf{K} \otimes \ell^2\left(\mathbb{Z}\right)^{\otimes m} \otimes \ell^2\left(\mathbb{N}_0\right)^{\otimes n}$. Define
$$\overline{u_{ij}}= U_{ij}\otimes I^{\otimes (m+n)},\, 1\leq i<j\leq m+n, \text{ and } \overline{s_i} = F_i\otimes S^*\otimes I^{\otimes (m+n-i)},\, 1\leq i\leq m+n.$$ 
Then, by direct computation on the standard basis vectors $k \otimes e_{n_1} \otimes \cdots \otimes e_{n_m} \otimes e_{m_1} \otimes \cdots \otimes e_{m_n}$ of $\mathcal{H}_0$, where $k \in \mathbf{K}$, $n_1,\dots,n_m \in \mathbb{Z}$, and $m_1,\dots,m_n \in \mathbb{N}_0$, it follows that the elements $\overline{u_{ij}}$ for $1\leq i<j\leq m+n$, and $\overline{s_i}$ for $1\leq i\leq m+n$ satisfy the defining relations of $C_X^{m,n}$. Hence, there exists a unique $^*$-homomorphism $\pi : C_X^{m,n} \to \mathcal{L}(\mathcal{H}_0)$ such that $\pi(u_{ij}) = \overline{u_{ij}}$ and $\pi(s_i) = \overline{s_i}$ for all relevant $i, j$. By Proposition~\ref{injective representation}, to show that $\pi$ is injective, it suffices to verify that for each $x \in X\cap \Omega_{m+n}$, the map $\pi_x : C_X^{m,n} \to C_x^{m,n}$ factors through $\pi$.

Let $x \in X\cap\Omega_{m+n}$, and let $x=e^{2\pi\mathrm{i}\Theta}$ where $\Theta$ is nondegenerate. Consider the proper two-sided ideal $\mathcal{I}_x:= \sum_{1\leq i<j\leq m+n} C_X^{m,n}(u_{ij} - x_{ij} I)$ of $C_X^{m,n}$. Then $\overline{\mathcal{I}_x} := \pi\left[\mathcal{I}_x\right]$ is a proper two-sided ideal in $C := \pi\left[C_X^{m,n}\right]$, since $x \in \sigma(\overline{u})$. Let $q: C \to C/\overline{\mathcal{I}_x}$ denote the quotient map. Then, the images
$$q(\overline{u_{ij}}) = x_{ij} I,\,1\leq i<j\leq m+n, \quad q(\overline{s_i}) =: \widetilde{s_i},\, 1\leq i\leq m+n$$
satisfy the defining relations of $C_x^{m,n}$. Thus, there exists a unique $^*$-homomorphism
\[
\iota : C_x^{m,n} \to C/\overline{\mathcal{I}_x}
\]
with $\iota(s_i) = q(\overline{s_i})$ for $1\leq i\leq m+n$. By construction, we have $\iota \circ \pi_x = q \circ \pi$,  which in particular shows that $\iota$ is surjective.

\medskip
\noindent\textbf{Case 1:} $n = 0$. Then, by \cite[Theorem 3.7]{Sla-1972aa}, the algebra $C_x^{m,0}$ is simple, since $\Theta$ is nondegenerate. Hence, $\iota$ is injective.

\medskip
\noindent\textbf{Case 2:} $n\neq 0$. Let $\mathcal{P}= \prod_{i=m+1}^{m+n}\left(I-s_i s_i^*\right)$. Suppose $\iota\left(\mathcal{P}\right) =0$. Then
        $$q\left(\prod_{i=m+1}^{m+n}\left(I-\overline{s_i}\;\overline{s_i}^*\right)\right) = \iota\circ\pi_x\left(\prod_{i=m+1}^{m+n}\left(I-s_i s_i^*\right)\right) = 0,$$
        which implies that $\displaystyle\prod_{i=m+1}^{m+n}\left(I-\overline{s_i}\;\overline{s_i}^*\right) \in\overline{\mathcal{I}_x}$. Since, each $F_i$ is a unitary operator, we have
        $$\prod_{i=m+1}^{m+n}\left(I-\overline{s_i}\;\overline{s_i}^*\right) = I\otimes \left(I-S^*S\right)^{\otimes n} = I\otimes p^{\otimes n} \in \overline{\mathcal{I}_x},$$
        so there exists $c_{ij}\in C$ for $1\leq i<j\leq m+n$ such that $\sum_{1\leq i<j\leq m+n} c_{ij}\left(\overline{u_{ij}} - x_{ij} I\right) = I\otimes p^{\otimes n}$.
        
        Define operators $E:\mathbf{K} \to \mathcal{H}_0$, and $Q: \mathcal{H}_0 \to \mathbf{K}$ by $E(k) = k\otimes e_0^{\otimes (m+n)}$ and
        $$Q\left(k\otimes e_{n_1}\otimes\cdots\otimes e_{n_m}\otimes e_{m_1}\otimes\cdots\otimes e_{m_n}\right) = \left(\delta_{n_1}\cdots\delta_{n_m}\right)\left(\delta_{m_1}\cdots\delta_{m_n}\right)k,$$
 respectively, for $k\in\mathbf{K}$, $n_1,\ldots, n_m\in\mathbb{Z}$, and $m_1,\ldots, m_n\in\mathbb{N}_0$. Evaluating on the elements of the form $k\otimes e_{n_1}\otimes\cdots\otimes e_{n_m}\otimes e_{m_1}\otimes\cdots\otimes e_{m_n}$, $k\in\mathbf{K}$, $n_1,\ldots,n_m\in\mathbb{Z}$, and $m_1,\ldots, m_n\in\mathbb{N}_0$, it follows that
        $$\sum_{1\leq i<j\leq m+n} \left(Qc_{ij} E\right)\left(U_{ij} - x_{ij} I\right) = I,$$
       This contradicts $x \in \sigma(\mathcal{U})$. Hence $\iota(\mathcal{P}) \neq 0$.

Let $B$ be the $C^*$-subalgebra of $C_x^{m,n}$ generated by $\mathcal{P} s_1, \dots, \mathcal{P} s_m$, and define $x_0 = \langle x_{ij} \rangle_{1 \leq i < j \leq m}$, with corresponding $\Theta_0 := \left\langle \theta_{ij}\right\rangle_{1\leq i<j\leq m}$. Then, there exists a surjective $^*$-homomorphism $\rho$ from the Higher dimensional noncommutative torus $A_{\Theta_0}$ to $B$. Since $\Theta_0$ is nondegenerate, it follows from \cite[Theorem 3.7]{Sla-1972aa} that $A_{\Theta_0}$ is simple. As $\iota|_B \circ \rho$ is a nonzero $^*$-homomorphism, it must be injective. Thus, $\iota|_B$ is injective, and by Corollary~\ref{faithful homomorphism m,n}, $\iota$ is injective. 

Therefore, $\iota$ is an isomorphism, and we have $\pi_x = \iota^{-1} \circ q \circ \pi$. Hence $\pi$ is injective. This completes the proof.
        \end{proof}
        
        \begin{rmrk}
      	Let $m, n \in \mathbb{N}_0$ with $m + n \geq 2$, and let
      	$$
      	\mathcal{S} := \left\{ \Theta = \left\langle \theta_{ij} \right\rangle_{1 \leq i < j \leq m+n} \in \mathbb{R}^{\binom{m+n}{2}} : \theta_{ij} \text{ are } \mathbb{Q}\text{-linearly independent} \right\}.
      	$$
      	Then $\left\{ e^{2\pi \mathrm{i} \Theta} : \Theta \in \mathcal{S} \right\} \subseteq \Omega_{m+n}^*$. By the multidimensional Kronecker approximation theorem, the set on the left is dense in $\mathbb{T}^{\binom{m+n}{2}}$. Hence $\Omega_{m+n}^*$ is dense in $\mathbb{T}^{\binom{m+n}{2}}$, and Theorem~\ref{F.rep m,n} in particular yields a faithful representation of $C_{X}^{m,n}$, where $X=\mathbb{T}^{\binom{m+n}{2}}$.
      	\end{rmrk}

        \begin{rmrk}
        Let $n\geq 2$. Let $X \subseteq \mathbb{T}^{\binom{n}{2}}$ be a closed set, and write each $x \in X$ as $x = \langle x_{ij} \rangle_{1 \leq i < j \leq n}$. For each $1 \leq i < j \leq n$, define the unitary operator
$$
U_{ij} : L^2(X) \to L^2(X), \quad (U_{ij} f)(x) = x_{ij} f(x),
$$
for all $f \in L^2(X)$ and $x \in X$. Then we have $\sigma(\mathcal{U}) = X$, where $\mathcal{U} = (U_{ij})_{1 \leq i < j \leq n}$.
        \end{rmrk}
        
        \begin{crlre}\label{F.rep.C2}
        	Let $X$ be a closed subset of $\mathbb{T}$, and let $\left\{x_m : m \in \mathbb{Z}\right\}$ be a dense subset of $X$. Define the operators $D: \ell^2(\mathbb{Z}) \to \ell^2(\mathbb{Z})$ and $f: \ell^2(\mathbb{Z}) \otimes \ell^2(\mathbb{N}_0) \to \ell^2(\mathbb{Z}) \otimes \ell^2(\mathbb{N}_0)$ by
        	\[
        	D(e_m) = x_m e_m \quad \text{and} \quad f(e_m \otimes e_n) =  \overline{x_m}^n e_m \otimes e_n,
        	\]
        	for $m \in \mathbb{Z}$ and $n \in \mathbb{N}_0$. Then, there exists a unique faithful $^*$-representation
        	\[
        	\pi: C_X^{0,2} \to \mathcal{L}\left(\ell^2(\mathbb{Z}) \otimes \ell^2(\mathbb{N}_0) \otimes \ell^2(\mathbb{N}_0)\right)
        	\]
        	such that
        	\[
        	\pi(u) = D \otimes I \otimes I, \quad \pi(s_1) = I \otimes S^* \otimes I, \quad \text{and} \quad \pi(s_2) = f \otimes S^*.
        	\]
        \end{crlre}
        
        \begin{proof}
        	The result follows from Theorem \ref{F.rep 0,n} by setting $U$ as the given unitary operator on $\ell^2\left(\mathbb{Z}\right)$.
        \end{proof}
        
        \begin{crlre}\label{F.rep.C3}
        	Let $\theta\in\mathbb{R}\setminus\mathbb{Q}$. Define the operators $D: \ell^2\left(\mathbb{Z}\right)\to \ell^2\left(\mathbb{Z}\right)$ and $f:\ell^2\left(\mathbb{Z}\right)\otimes\ell^2\left(\mathbb{N}_0\right)\to \ell^2\left(\mathbb{Z}\right)\otimes\ell^2\left(\mathbb{N}_0\right)$ by
        	$$D\left(e_m\right) = e^{2\pi\mathrm{i}m\theta} e_m \text{ and } f(e_m\otimes e_n)= e^{-2\pi\mathrm{i}mn\theta} e_m\otimes e_n,$$
        	for $m\in\mathbb{Z}$ and $n\in\mathbb{N}_0$. Then, there exists a unique faithful $^*$-representation 
        	$$\pi : C_{\mathbb{U}}^{0,2}\to \mathcal{L}\left(\ell^2\left(\mathbb{Z}\right)\otimes\ell^2\left(\mathbb{N}_0\right)\otimes \ell^2\left(\mathbb{N}_0\right)\right)$$
        	such that 
        	\[
        	\pi(u) = D \otimes I \otimes I, \quad \pi(s_1) = I \otimes S^* \otimes I, \quad \text{and} \quad \pi(s_2) = f \otimes S^*.
        	\]
        \end{crlre}

        \begin{proof}
        	Since $\left\{e^{2\pi\mathrm{i}m\theta} : m \in \mathbb{Z}\right\}$ is dense in $\mathbb{T}$, the result follows directly from Corollary \ref{F.rep.C2}.
        \end{proof}
        
        Let $\mathcal{U}=\{U_{ij}:1\leq i <j \leq m+n\}$ be a tuple of commuting unitaries acting on $\clh$.  Let $X$ be the joint spectrum of $\mathcal{U}$.  The question is; is there a tuple $\mathcal{T}=(T_1,T_2,\cdots , T_{m+n})$ of $\mathcal{U}$-doubly twisted $(m,n)$-isometries  on $\clh$  such that the canonical  surjective homomorphism  $\phi:C_{X}^{m,n} \rightarrow C^*(\mathcal{T})$ sending $u_{ij} \mapsto U_{ij}$ and $s_i\mapsto T_i$ is an isomorphism?  If such a tuple exists then we call it a universal tuple of $\mathcal{U}$-twisted isometries and in such a case, $C^*(\mathcal{T})$ is denoted by  $C_{\mathcal{U}}^{m,n}$.  
        \bxmpl 
        \begin{enumerate}[(i)]
        	\item    Let $\clh = \bbc$ and $U=e^{2\pi \mathrm{i}  \theta}$.  If $\theta \notin \bbz$, then there does not exist any pair of isometries satisfying the defining relations of $C_{\mathbb{U}}^{0,2}$. Hence in this case, $C_{\mathcal{U}}^{0,2}$ does not exist. 
        	\item Fix $q \in \bbn$ and $\theta =\frac{p}{q}$ with $(p,q)=1$. Let $\clh = \bbc^q$ and $U=e^{2\pi \mathrm{i}  \theta}I_q$.  Then there are infinitely many pairs of unitaries satisfying (\ref{R2}). However, $C_{\mathcal{U}}^{2,0}$ does not exist.
        \end{enumerate}
        \exmpl 
        Our objective is to find a necessary and sufficient conditions on $\mathcal{U}$ under which a universal tuple of $\mathcal{U}$-doubly twisted $m,n)$-isometries. We begin with a lemma.
        \blmma 
        Let $\mathcal{U}=\{U_{ij}:1\leq i <j \leq m+n\}$ be a tuple of commuting unitaries defined on a infinite dimensional separable Hilbert space $\clh$. Let $X$ be the joint spectrum of $\mathcal{U}$.  Let $(T_1,T_2,\cdots T_{m+n})$ be a universal tuple of $\mathcal{U}$-doubly twisted $(m,n)$-isometries. Then we have 
        $$\sigma^{I_{m+n}\setminus I_n}(\mathcal{U})=X.$$
        \elmma 
        \prf  Let  $A_0= I_{m+n}\setminus I_n$. Let $C=C^*(\{T_i: 1\leq i\leq m+n\})$.  By von-Nuemann Wold decomposition, we have 
        $$T_i=\oplus_{A\subset A_0}T_i^A, \, \mbox{ where } T_i^A=\restr{T_i}{\clh_A} \mbox{ for all } A \subset A_0\mbox{ and } 1\leq i \leq m+n.$$
        Let  $C^A=C^*(\{T_i^A: 1\leq i\leq m+n\})$.   Let $\mathcal{I}$ be the ideal of $C$ generated by $\prod_{i=m+1}^{m+n}(1-T_iT_i^*)$ and $\mathcal{I}^{A_0}$ be the ideal  of $C^{A_0}$ generated by $\prod_{i=m+1}^{m+n}(1-T_i^{A_0}
        (T_i^{A_0})^*)$.
        Then we have
        $$\mathcal{I}=\mathcal{I}^{A_0}\oplus 0 \oplus 0 \cdots 0,$$
        Therefore, $\widehat{\mathcal{I}}\cong \widehat{\mathcal{I}^{A_0}}$. In fact, from \cite{Arv-1976aa}, it follows that  any irreducible  representation $\pi$ of $C$ which is not vanishing on $\mathcal{I}$  will be given by an irreducible representation $\pi^{A_0}$ of $C^{A_0}$ in the following way.
        $$\pi(T_i)=\pi^{A_0}(T_i^{A_0}), \, \mbox{ and } \, \pi(U_{ij})=\pi^{A_0}(U_{ij}^{A_0}).$$
        This shows that  in any such irreducible representation of $C$, we have  $\{\pi(U_{ij})\} \in \sigma^{A_0}.$ 
        
      Now, take  $ \Theta \in X$. Consider the irreducible representation  $\pi_{(I_m,\rho)}^{\Theta}$ of $C$.  Since $\pi_{(I_m,\rho)}^{\Theta}$  does not vanish on $\mathcal{I}$, it follows that $\Theta \in \sigma^{A_0}(\mathcal{U})$, proving that $X \subset \sigma^{A_0}(\mathcal{U})$.  Since $U_{ij}$ reduces $\clh^A$ for $A \subset A_0$, we have 
        $$X= \sigma(\mathcal{U})=\cup_{A\subset A_0} \sigma^A(\mathcal{U}) \supset \sigma^{A_0}(\mathcal{U}), $$
        which proves the claim. 
        \qed

       \bthm  \label{infinite multiplicity}
       Let $\mathcal{U}=\{U_{ij}:1\leq i <j \leq m+n\}$ be a tuple of commuting unitaries defined on a infinite dimensional separable Hilbert space $\clh$. Let $X$ be the joint spectrum of $\mathcal{U}$. If
      there exists a universal tuple of $\mathcal{U}$-doubly twisted $(m,n)$-isometries,  then 
      $$M_{\mathcal{U}}(X)=\infty,$$
      where $M_{\mathcal{U}}$ is the spectral muliplicity of  $\mathcal{U}$. Moreover, if $X\cap\Omega_{m+n}^*$ is dense in $X$, and $M_{\mathcal{U}}(X)=\infty$, then  a universal tuple of $\mathcal{U}$-twisted isometries exist.
       \ethm  
       \prf Let $\mathcal{T}=(T_1,\cdots , T_{m+n})$ be a universal tuple of $\mathcal{U}$-twisted isometries   such that $\mathcal{T}   \leadsto C_{X}^{m,n}$. 
       \textbf{Claim:} In the Wold-decomposition of $(T_1,\cdots , T_{m+n})$, the wandering space $W_{A_0}\neq \{0\}$, where $A_0=\{m+1,m+2, \cdots m+n\}$.\\
       \textbf{proof of the claim:} It is equivalent to show that $\prod_{i=m+1}^{m+n}(1 - T_i T_i^*) \neq 0$, and for that, it suffices to find a representation $\pi$ of $C^{m,n}_X$ such that $\pi\left(\prod{i=m+1}^{m+n}(1 - T_i T_i^*)\right) \neq 0$. This follows immediately from Theorem \ref{representations1}. \\
      \noindent Now we have 
      $$\clh_A=\bigoplus_{(r_{m+1},\cdots r_{m+n})\in \bbn_0^n}\, \prod_{i=m+1}^{m+n}T_i^{r_i}W_A\cong (\ell^2(\bbn_0))^{\otimes n}\otimes W_A.$$
      Moreover, $U_{ij}^A$ reduces $W_A$  and the joint spectrum of $\mathcal{U}^A$ is $X$. This shows that  $M_{\mathcal{U}}(X)=\infty$. To prove the converse,  assume that $M_{\mathcal{U}}(X)=\infty$. Then by the Hahn-Hellinger theorem, there exists a measure $(X, B_X, \mu_{\infty} )$ such that 
      $$ \clh \cong L^2(X) \otimes \ell^2(\bbn_0), \quad \mbox{ and } \quad  U_{ij} \cong M_{x_{ij}} \otimes 1,  \mbox { for } 1 \leq i <j \leq m+n.$$ Now the claim follows from Theorem \ref{F.rep m,n}. 
       \qed

       \bppsn \label{center fibre}
       Let  $m, n \in \bbn_0$ and $\Theta_{[m]} \in \Omega_m^*$. Then $C_{\Theta}^{m,n}$ is a primary $C^*$-algebra, and hence  $\mathcal{Z}(C_{\Theta}^{m,n})=\bbc\mathds{1}$.
       \eppsn 
       \prf Since $\Theta_{[m]}$  is nondegenerate, there exists a faithful irreducible representation $\rho$ of $\cla_{\Theta_{[m]}}^m$. By Theorem $2.4$ and Theorem $2.7$ in \cite{BhaSau-2023aa}, the representation $\pi_{(I_m,\rho)}^{\Theta}$ is irreducible and faithful, hence the claim. 
       \qed 
       
       \bthm  \label{center main} Let $X$ be a closed subset of $\bbbt^{m+n \choose 2}$ and let $A$ be a unital continuous $C(X)$-algebra with the structure map $\beta: C(X) \rightarrow \mathcal{Z}(A)$. Assume that there exists a dense set $\bigwedge$ such that $\mathcal{Z}(A_x)=\bbc\mathds{1}$ for $x \in \bigwedge$. Then one has
       $$\mathcal{Z}(A)=\beta(C(X)).$$
       \ethm 
       \prf Let $B=\beta(C(X))$.  Take  $T \in \mathcal{Z}(A)$. For $x \in X$, define $B_x=\pi_x(B)$. Since $\mathcal{Z}(A_x)=\bbc \mathds{1}$ for $x \in  \bigwedge$, it follows that $T_x \in B_x$. \\
      \textbf{Claim:} $T_x \in B_x, \mbox{ for all } x \in X$. \\
      \textbf{proof of the claim:} Since $ \bigwedge$ is dense in $X$, we get a sequence $\{x_n\}_{n=1}^{\infty}$ such that $x_n \rightarrow x$ as $ n \rightarrow \infty$.  Let $T_{x_n}=w_n\mathds{1}$.  Since $\pi_{x_n}$ is a homomorphism, we have $\|\pi_{x_n}\|\leq 1$ and hence $|w_n|=\|T_{x_n}\|\leq \|T\|$. By Bolzano Weierstrass Theorem, we get a  subsequence $\{w_{n_k}\}_{k=1}^{\infty}$ such that  $w_{n_k} \rightarrow w$ for some $w \in \bbc$ and the subspace topology on $\{w_{n_k}:1 \leq k < \infty\} $ is discrete.  Hence there exists $f \in C(X)$ such that 
      $$f(x_{n_k})=w_{n_k},\, \mbox{ for all } k \in \bbn.$$
      By continuous functional calculus, we have $f(\mathcal{U})\in B$ and $\pi_y(f(\mathcal{U}))=f(y)\mathds{1}$ for all $y \in X$. Thus, we have 
      $$\|\pi_x(T-f\mathds{1})\|=\lim_{k \rightarrow \infty} \|\pi_{x_{n_k}}(T-f\mathds{1})\|=\lim_{k \rightarrow \infty} \|T_{x_{n_k}}-w_{n_k}\mathds{1}\|=0.$$
      This proves that $T_x=f(x)\mathds{1}=w\mathds{1}\in B_x$.   \\
      Using this, together with part(iii) of  Lemma $2.1$ in \cite{Dad-2009aa}, we conclude  that $T \in B$, thus proving the result.
      \qed 
      \bcrlre Let $m,n \in \bbn_0$ with $m+n\geq 2$. Then one has the following. 
      $$\mathcal{Z}(C_X^{m,n})=C^*(\{u_{ij}:1\leq i <j \leq m+n\}).$$
      \ecrlre
      \prf 
      By Proposition \ref{continuous algebra}, it follows that $C_X^{m,n}$ is a unital  continuous $C(X)$-algebra with fibre $\pi_{\Theta}(C_X^{m,n})=C_{\Theta}^{m,n}$. Moreover,  from Proposition \ref{center fibre}, one can conclude  that $\mathcal{Z}(C_{\Theta}^{m,n})=\bbc\mathds{1}$, provided $\Theta \in \Omega_{m+n}^*$. Since  $ \Omega_{m+n}^*$ is  dense in $\Omega_{m+n}$, we get the claim from Theorem \ref{center main}. 
      \qed
       
       \section{$K$-groups of $C_{\mathbb{U}}^{m,n}$}\label{K-theory}
       In this section, we compute the $K$-groups of the $C^*$-algebra $C_{\mathbb{U}}^{m,n}$ for all $m,n \in \mathbb{N}_0$ with $m+n \geq 1$.
       
       \begin{dfn}\label{inducing unitary}
       	Let $m,n \in \mathbb{N}_0$ with $m+n \geq 1$. Define the universal $C^*$-algebra $A_{\mathbb{U}}^{m,n}$ as follows:
       	
       	\begin{itemize}
       		\item If $(m,n) = (1,0)$, then $A_{\mathbb{U}}^{m,n}$ is generated by elements $s_1$ and $u_{12}$ satisfying
       		$$
       		s_1^* s_1 = s_1 s_1^* = u_{12}^* u_{12} = u_{12} u_{12}^* = 1, \quad s_1 u_{12} = u_{12} s_1.
       		$$
       		
       		\item If $(m,n) = (0,1)$, then $A_{\mathbb{U}}^{m,n}$ is generated by elements $s_1$ and $u_{12}$ satisfying
       		$$
       		s_1^* s_1 = u_{12}^* u_{12} = u_{12} u_{12}^* = 1, \quad s_1 u_{12} = u_{12} s_1.
       		$$
       		
       		\item If $m+n > 1$, then $A_{\mathbb{U}}^{m,n}$ is generated by elements $s_1, \ldots, s_{m+n}$ and a set\linebreak $\mathbb{U} := \{ u_{ij} : 1 \leq i < j \leq m+n+1 \}$
       		satisfying
       		\begin{itemize}
       			\item[$\mathrm{(i)}$] $s_i s_i^* = 1$ for $1 \leq i \leq m$, and $s_i^* s_i = 1$ for $1 \leq i \leq m+n$,
       			\item[$\mathrm{(ii)}$] $u_{ij}^* u_{ij} = u_{ij} u_{ij}^* = 1$ for $1 \leq i < j \leq m+n+1$,
       			\item[$\mathrm{(iii)}$] $u_{ij} s_k = s_k u_{ij}$ for $1 \leq i < j \leq m+n+1$, $1 \leq k \leq m+n$,
       			\item[$\mathrm{(iv)}$] $u_{ij} u_{pq} = u_{pq} u_{ij}$ for $1 \leq i < j \leq m+n+1$, $1 \leq p < q \leq m+n+1$,
       			\item[$\mathrm{(v)}$] $s_i^* s_j = u_{ij}^* s_j s_i^*$ for $1 \leq i < j \leq m+n$.
       		\end{itemize}
       	\end{itemize}
       \end{dfn}
       
       For the identity automorphism $\id$ on the algebra $C_{\mathbb{U}}^{m,n}$, by \cite[Corollary~9.4.23]{GioKerPhiTom-2017aa}, $C_{\mathbb{U}}^{m,n} \otimes C(\mathbb{T}) = C_{\mathbb{U}}^{m,n} \rtimes_{\id} \mathbb{Z}$ is the universal $C^*$-algebra generated by $C_{\mathbb{U}}^{m,n}$ and a unitary $w$ satisfying $s_i^{m,n} w = w s_i^{m,n}$ for $1 \leq i \leq m+n$, and $u_{ij} w = w u_{ij}$ for $1 \leq i < j \leq m+n$. Iterating this construction and adjoining a unitary at each step, we obtain an isomorphism
       $$
       \tau: A_{\mathbb{U}}^{m,n} \longrightarrow C_{\mathbb{U}}^{m,n} \otimes C\left(\mathbb{T}\right)^{\otimes (m+n)}
       $$
       given by
       $$
       s_l^{m,n} \mapsto s_l^{m,n} \otimes 1^{\otimes (m+n)}, \quad u_{ij} \mapsto u_{ij} \otimes 1^{\otimes (m+n)}, \quad  u_{l,m+n+1} \mapsto 1^{\otimes l} \otimes z \otimes 1^{\otimes (m+n-l)},
       $$
       for $1 \leq l \leq m+n$ and $1 \leq i < j \leq m+n$, where $z$ denotes the identity function $z \mapsto z$, $z \in \mathbb{T}$.
       
       Moreover, it is easy to verify for all $m,n \in \mathbb{N}_0$ with $m+n \geq 1$ that the map $\tau_{m,n} : A_{\mathbb{U}}^{m,n} \to A_{\mathbb{U}}^{m,n}$ defined by
       $$
       \tau_{m,n}\left(u_{ij}\right) = u_{ij} \quad \text{for } 1 \leq i < j \leq m+n+1, \quad \tau_{m,n}(s_i^{m,n}) =
       \begin{cases}
       	u_{i,m+n+1}^* s_i^{m,n} & \text{if } 1 \leq i \leq m, \\
       	u_{i,m+n+1} s_i^{m,n} & \text{if } m+1 \leq i \leq m+n.
       \end{cases}
       $$
       is a well-defined $^*$-automorphism of $A_{\mathbb{U}}^{m,n}$.

       \bppsn\label{cross product unitary}
       Let $m,n \in \mathbb{N}_0$ with $m+n\geq 1$. Then there exists a canonical $^*$-isomorphism $A_{\mathbb{U}}^{m,n} \rtimes_{\tau_{m,n}} \mathbb{Z} \cong C_{\mathbb{U}}^{m+1,n}$.
       \eppsn
       
       \prf
       By \cite[Corollary~9.4.23]{GioKerPhiTom-2017aa}, the crossed product $A_{\mathbb{U}}^{m,n} \rtimes_{\tau_{m,n}} \mathbb{Z}$ is the universal $C^*$-algebra generated by $A_{\mathbb{U}}^{m,n}$ and a unitary $w$ satisfying
       $$
       w s_i^{m,n} w^* = \tau_{m,n}(s_i^{m,n}) \quad \text{for } 1 \leq i \leq m+n, \quad \text{and} \quad w u_{ij} w^* = u_{ij} \quad \text{for } 1 \leq i < j \leq m+n+1.
       $$
       These relations coincide with those defining $C_{\mathbb{U}}^{m+1,n}$, where the additional generator $s_{m+1}^{m+1,n}$ corresponds to $w$. Define a map $\phi_{m,n} : A_{\mathbb{U}}^{m,n} \rtimes_{\tau_{m,n}} \mathbb{Z} \to C_{\mathbb{U}}^{m+1,n}$ on generators by
       $$
       \phi_{m,n}(w) = s_{m+1}^{m+1,n}, \quad
       \phi_{m,n}(s_i^{m,n}) =
       \begin{cases}
       	s_i^{m+1,n} & \text{if } 1 \leq i \leq m, \\
       	s_{i+1}^{m+1,n} & \text{if } m+1 \leq i \leq m+n,
       \end{cases}
       $$
       and
       $$
       \phi_{m,n}(u_{ij}) =
       \begin{cases}
       	u_{ij} & \text{if } 1 \leq i < j \leq m, \\
       	u_{i,m+1} & \text{if } 1 \leq i \leq m,\, j = m+n+1, \\
       	u_{i,j+1} & \text{if } 1 \leq i \leq m,\, m+1 \leq j \leq m+n, \\
       	u_{i+1,j+1} & \text{if } m+1 \leq i < j \leq m+n, \\
       	u_{m+1,i+1} & \text{if } m+1 \leq i \leq m+n,\, j = m+n+1.
       \end{cases}
       $$
       It is straightforward to verify that $\phi_{m,n}$ preserves all defining relations. By the universal property, $\phi_{m,n}$ extends to a $^*$-isomorphism.
       \qed
       
       \begin{ppsn}\label{K1 general 1}
       	Let $m \in \mathbb{N}$. Then the $K$-groups of the $C^*$-algebra $C_{\mathbb{U}}^{m,0}$ are given by
       	$$
       	K_0\left(C_{\mathbb{U}}^{m,0}\right) \cong K_1\left(C_{\mathbb{U}}^{m,0}\right) \cong \mathbb{Z}^{2^{\left(m + \binom{m}{2} - 1\right)}}.
       	$$
       \end{ppsn}
       \prf
       We proceed by induction on $m$.
       
       \textbf{Base case:} For $m = 1$, we have $C_{\mathbb{U}}^{1,0} \cong C(\mathbb{T})$. Hence, $K_0\left(C_{\mathbb{U}}^{1,0}\right) \cong K_1\left(C_{\mathbb{U}}^{1,0}\right) \cong \mathbb{Z}$, which agrees with the formula.
       
       \textbf{Inductive step:} Assume the result holds for some $m \geq 1$, i.e.,
       $$
       K_0\left(C_{\mathbb{U}}^{m,0}\right) \cong K_1\left(C_{\mathbb{U}}^{m,0}\right) \cong \mathbb{Z}^{2^{\left(m + \binom{m}{2} - 1\right)}}.
       $$
       Consider $A_{\mathbb{U}}^{m,0} \cong C_{\mathbb{U}}^{m,0} \otimes C(\mathbb{T})^{\otimes m}$, and note that each tensor by $C\left(\mathbb{T}\right)$ is isomorphic to a crossed product by $\mathbb{Z}$ via the identity automorphism. Since the $K$-groups of $C_{\mathbb{U}}^{m,0}$ are free abelian, by repeatedly applying Pimsner-Voiculescu six-term exact sequence for the respective crossed product \cite[Theorem 2.4]{PimVoi-1980aa}, we obtain,
       
       $$K_{0}\left(A_{\mathbb{U}}^{m,0}\right) \cong K_{1}\left(A_{\mathbb{U}}^{m,0}\right)\cong K_0\left(C_{\mathbb{U}}^{m,0}\right)^{\oplus 2^{m-1}} \oplus K_1\left(C_{\mathbb{U}}^{m,0}\right)^{\otimes 2^{m-1}} = \mathbb{Z}^{2^{2m+\binom{m}{2}-1}}.$$
       
       Consider the crossed product $A_{\mathbb{U}}^{m,0} \rtimes_{\tau_{m,0}} \mathbb{Z}$ as in Proposition \ref{cross product unitary}, and the associated Pimsner--Voiculescu six-term exact sequence:
       
       \begin{center} 
       	\begin{tikzpicture}[node distance=1cm,auto]
       		\tikzset{myptr/.style={decoration={markings,mark=at position 1 with %
       					{\arrow[scale=2,>=stealth]{>}}},postaction={decorate}}}
       		\node (A){$ K_0\left(A_{\mathbb{U}}^{m,0}\right)$};
       		\node (Up)[node distance=2.5cm, right of=A][label=above:$\id_*-(\tau_{m,0})_{*}^{-1}$]{};
       		\node (B)[node distance=5cm, right of=A]{$K_0\left(A_{\mathbb{U}}^{m,0}\right)$};
       		\node (Up)[node distance=2cm, right of=B][label=above:$i_*$]{};
       		\node (C)[node distance=5cm, right of=B]{$K_0\left(A_{\mathbb{U}}^{m,0}\rtimes_{\tau_{m,0}} \bbz\right)$};
       		\node (D)[node distance=2cm, below of=C]{$K_1\left(A_{\mathbb{U}}^{m,0}\right)$};
       		\node (E)[node distance=5cm, left of=D]{$K_1\left(A_{\mathbb{U}}^{m,0}\right)$};
       		\node (F)[node distance=5cm, left of=E]{$K_1\left(A_{\mathbb{U}}^{m,0}\rtimes_{\tau_{m,0}} \bbz\right)$};
       		\node (Below)[node distance=2.3cm, left of=D][label=below:$\id_*-(\tau_{m,0})_{*}^{-1}$]{};
       		\node (Below)[node distance=2cm, left of=E][label=below:$i_*$]{};
       		\draw[myptr](A) to (B);
       		\draw[myptr](B) to (C);
       		\draw[myptr](C) to node{{ $\partial$}}(D);
       		\draw[myptr](D) to (E);
       		\draw[myptr](E) to (F);
       		\draw[myptr](F) to node{{ $\delta$}}(A);
       	\end{tikzpicture}
       \end{center}  
       
       Define for each $t \in [0,1]$ an automorphism $\tau_{m,0}^{(t)}$ on $A_{\mathbb{U}}^{m,0}$ by
       $$
       \tau_{m,0}^{(t)}(u_{ij}) = u_{ij} \text{ for } 1 \leq i < j \leq m+1, \quad \tau_{m,0}^{(t)}(s_i) = u_{i,m+1}^{-t} s_i \text{ for } 1 \leq i \leq m.
       $$
       Then $\tau_{m,0}^{(0)} = \mathrm{id}$ and $\tau_{m,0}^{(1)} = \tau_{m,0}$, so the continuous family $\tau_{m,0}^{(t)}$, $t\in [0,1]$ defines a homotopy between $\mathrm{id}$ and $\tau_{m,0}$. This implies that $\mathrm{id}_* - (\tau_{m,0})_*^{-1} = 0$. By exactness of the Pimsner-Voiculescu sequence and freeness of the $K$-groups, we obtain
       \begin{align*}
       	K_0\left(A_{\mathbb{U}}^{m,0} \rtimes_{\tau_{m,0}} \mathbb{Z}\right) \cong K_1\left(A_{\mathbb{U}}^{m,0} \rtimes_{\tau_{m,0}} \mathbb{Z}\right) \cong K_0\left(A_{\mathbb{U}}^{m,0}\right) \oplus K_1\left(A_{\mathbb{U}}^{m,0}\right) &\cong \mathbb{Z}^{2^{2m + \binom{m}{2}}}\\
       	&= \mathbb{Z}^{2^{(m+1) + \binom{m+1}{2} - 1}}.
       \end{align*}
       Since $C_{\mathbb{U}}^{m+1,0} \cong A_{\mathbb{U}}^{m,0} \rtimes_{\tau_{m,0}} \mathbb{Z}$, the formula holds for $m+1$. By induction, the result follows for all $m \in \mathbb{N}$.
       \qed
       
       Let $A$ be a unital $C^*$-algebra, and let $\alpha$ be an automorphism of $A$. Consider the $C^*$-subalgebra $\mathcal{T}\left(A, \alpha\right)$ of $\left(A \rtimes_\alpha \mathbb{Z}\right) \otimes \mathcal{T}$ generated by $A \otimes 1$ and $w \otimes v$, where $w$ denotes the adjoint unitary in $A \rtimes_\alpha \mathbb{Z}$, and $v$ is an isometry generating the Toeplitz algebra $\mathcal{T}$. Then, as shown in \cite[Page~98]{PimVoi-1980aa}, there exists a short exact sequence
       $$
       0 \longrightarrow \mathcal{K} \otimes A \xrightarrow{i} \mathcal{T}\left(A, \alpha\right) \xrightarrow{\psi} A \rtimes_\alpha \mathbb{Z} \longrightarrow 0.
       $$
       
       The following proposition will play a key role in the computation of the $K$-groups for all $m,n$.
       
       \begin{ppsn}\label{Toeplitz identification}
       	Let $m,n \in \mathbb{N}_0$ with $m+n \geq 1$. Then there exists a $^*$-isomorphism $\phi : C_{\mathbb{U}}^{m,n+1} \longrightarrow \mathcal{T}\left(A_{\mathbb{U}}^{m,n}, \tau_{m,n}\right)$ satisfying $\phi(u_{ij}) = u_{ij} \otimes 1 \quad \text{for all } 1 \leq i < j \leq m+n+1$, and for $1 \leq i \leq m+n+1$,
       	$$\phi(s_i^{m,n+1}) =
       	\begin{cases}
       		s_i^{m+1,n}\otimes 1 & \text{if } i \neq m+1, \\
       		s_i^{m+1,n} \otimes v & \text{if } i = m+1.
       	\end{cases}$$
       \end{ppsn}
       
       \begin{proof}
       	\begin{figure}[h]
       		\centering
       		\[
       		\begin{tikzcd}[
       			column sep=large,
       			row sep=large,
       			execute at end picture={
       				% First \circlearrowright at (x1,y1)
       				\node at (-0.4, 0.0) {\scalebox{3.0}{$\circlearrowright$}};
       			}
       			]
       			C_{\mathbb{U}}^{m,n+1}
       			\arrow[r, "\phi"] 
       			\arrow[d, "\pi_x"']        			
       			& \mathcal{T}\left(A_{\mathbb{U}}^{m,n},\tau_{m,n}\right) 
       			\arrow[d, "\pi_x^{(1)}\otimes id|_{\mathcal{T}\left(A_{\mathbb{U}}^{m,n},\tau_{m,n}\right)}"] 
       			\\
       			C_x^{m,n+1} 
       			\arrow[r, "\phi_x", "\sim"'] 
       			& C
       		\end{tikzcd}
       		\]
       		\caption{Diagram corresponding to Proposition \ref{Toeplitz identification}}
       		\label{fig:commutative_diagram1}
       	\end{figure}
       	
       	It is immediate from the definition that the map $\phi$ is a surjective $^*$-homomorphism. To prove injectivity, by Proposition~\ref{injective representation}, it suffices to show that for each nondegenerate point $x \in \mathbb{T}^{\binom{m+n+1}{2}}$, the canonical representation $\pi_x : C_{\mathbb{U}}^{m,n+1} \to C_x^{m,n+1}$ factors through $\phi$.
       	
       	Let $\pi_x^{(1)} : C_{\mathbb{U}}^{m+1,n} \to C_x^{m+1,n}$ be the canonical surjection. Define $C := \left(\pi_x^{(1)} \otimes \mathrm{id}\right)\left[ \mathcal{T}(A_{\mathbb{U}}^{m,n}, \tau_{m,n}) \right]$. Note that $C$ is the $C^*$-subalgebra of $C_x^{m+1,n} \otimes \mathcal{T}$ generated by the elements $s_i^{m+1,n} \otimes 1$ for $i \neq m+1$, and $s_{m+1}^{m+1,n} \otimes v$. Define a $^*$-homomorphism $\phi_x : C_x^{m,n+1} \to C$ by
       	\[
       	\phi_x\left(s_i^{m,n+1}\right) =
       	\begin{cases}
       		s_i^{m+1,n} \otimes 1 & \text{if } i \neq m+1, \\
       		s_i^{m+1,n} \otimes v & \text{if } i = m+1,
       	\end{cases}
       	\]
       	for $1 \leq i \leq m+n+1$. Let $\mathcal{P} := \displaystyle\prod_{i=m+1}^{m+n+1} \left(1 - s_i^{m,n+1}\left(s_i^{m,n+1}\right)^*\right)$, and let $D$ be the $C^*$-subalgebra of $C_x^{m,n+1}$ generated by the elements $\mathcal{P}s_i$, for $1 \leq i \leq m$. We compute:
       	$$\phi_x(\mathcal{P}) =
       	\begin{cases}
       		1 \otimes (1 - vv^*) & \text{if } n = 0, \\
       		\left( \displaystyle\prod_{i = m+2}^{m+n+1} \left(1 - s_i^{m+1,n}\left(s_i^{m+1,n}\right)^*\right) \right) \otimes (1 - vv^*) & \text{if } n \neq 0,
       	\end{cases}$$
       	which is nonzero. Hence, $\phi_x|_D$ is nonzero. Since $x$ is nondegenerate, it follows that $\phi_x|_D$ is faithful. By Corollary~\ref{faithful homomorphism m,n}, $\phi_x$ is injective, and thus an isomorphism.
       	
       	Therefore, the map $\pi_x$ factors as $\pi_x = \phi_x^{-1} \circ (\pi_x^{(1)} \otimes \mathrm{id}) \circ \phi$ (see Figure \ref{fig:commutative_diagram1}), which shows that $\phi$ is injective. Hence, $\phi$ is a $^*$-isomorphism.
       \end{proof}
       
       \begin{thm}
       	Let $m,n\in\mathbb{N}_0$ with $m+n\geq 2$. Then the $K$-groups of the $C^*$-algebra $C_{\mathbb{U}}^{m,n}$ are given by
       	$$K_0\left(C_{\mathbb{U}}^{m,n}\right) \cong K_1\left(C_{\mathbb{U}}^{m,n}\right) \cong \mathbb{Z}^{2^{m+\binom{m+n}{2}-1}}.$$
       \end{thm}  
       
       \begin{proof}
       	We proceed by induction on $n \in \mathbb{N}_0$.
       	
       	\textbf{Base case:} For $n = 0$, the assertion follows from Proposition~\ref{K1 general 1}.
       	
       	\textbf{Inductive step:} Assume that for some $n \geq 0$ and all $m \in \mathbb{N}_0$ with $m + n \geq 2$, we have
       	\[
       	K_0\left(C_{\mathbb{U}}^{m,n}\right) \cong K_1\left(C_{\mathbb{U}}^{m,n}\right) \cong \mathbb{Z}^{2^{m + \binom{m+n}{2} - 1}}.
       	\]
       	We now show that for all $m \in \mathbb{N}_0$ with $m+n \geq 1$, it holds that
       	\[
       	K_0\left(C_{\mathbb{U}}^{m,n+1}\right) \cong K_1\left(C_{\mathbb{U}}^{m,n+1}\right) \cong \mathbb{Z}^{2^{m + \binom{m+n+1}{2} - 1}}.
       	\]
       	By Proposition~\ref{Toeplitz identification}, we have $C_{\mathbb{U}}^{m,n+1} \cong \mathcal{T}\left(A_{\mathbb{U}}^{m,n}, \tau_{m,n}\right)$, and by Lemma~2.3 together with the proof of Theorem~2.4 in \cite{PimVoi-1980aa}, we obtain $K_i\left(\mathcal{T}\left(A_{\mathbb{U}}^{m,n}, \tau_{m,n}\right)\right) \cong K_i\left(A_{\mathbb{U}}^{m,n}\right), \, i=0,1$. Hence,
       	\[
       	K_i\left(C_{\mathbb{U}}^{m,n+1}\right) \cong K_i\left(A_{\mathbb{U}}^{m,n}\right), \quad i=0,1.
       	\]
       	
       	\textit{\textbf{Case 1:} $m+n = 1$.} Then $(m,n) = (0,1)$ or $(1,0)$.
       	
       	If $(m,n) = (0,1)$, then $A_{\mathbb{U}}^{0,1} \cong \mathcal{T} \otimes C\left(\mathbb{T}\right)$, so
       	\[
       	K_i\left(C_{\mathbb{U}}^{0,2}\right) \cong K_i\left(A_{\mathbb{U}}^{0,1}\right) \cong K_0\left(\mathcal{T}\right) \oplus K_1\left(\mathcal{T}\right) \cong \mathbb{Z}, \quad i = 0,1.
       	\]
       	
       	If $(m,n) = (1,0)$, then $A_{\mathbb{U}}^{1,0} \cong C\left(\mathbb{T}\right) \otimes C\left(\mathbb{T}\right)$, so
       	\[
       	K_i\left(C_{\mathbb{U}}^{1,1}\right) \cong K_i\left(A_{\mathbb{U}}^{1,0}\right) \cong K_0\left(C\left(\mathbb{T}\right)\right) \oplus K_1\left(C\left(\mathbb{T}\right)\right) \cong \mathbb{Z}^2, \quad i = 0,1.
       	\]
       	In both cases, the result agrees with the stated formula.
       	
       	\textit{\textbf{Case 2:} $m+n \geq 2$.} Then $A_{\mathbb{U}}^{m,n} \cong C_{\mathbb{U}}^{m,n} \otimes C\left(\mathbb{T}\right)^{\otimes (m+n)}$. By repeated application of the Pimsner-Voiculescu six-term exact sequence for crossed products of the form $A \rtimes_{\mathrm{id}} \mathbb{Z} \cong A \otimes C\left(\mathbb{T}\right)$, we get
       	\[
       	K_i\left(C_{\mathbb{U}}^{m,n+1}\right) \cong K_i\left(A_{\mathbb{U}}^{m,n}\right) \cong K_0\left(C_{\mathbb{U}}^{m,n}\right)^{\oplus 2^{m+n}} \oplus K_1\left(C_{\mathbb{U}}^{m,n}\right)^{\oplus 2^{m+n}}, \quad i = 0,1.
       	\]
       	Using the induction hypothesis, $K_0\left(C_{\mathbb{U}}^{m,n}\right) \cong K_1\left(C_{\mathbb{U}}^{m,n}\right) \cong \mathbb{Z}^{2^{m + \binom{m+n}{2} - 1}}$, we obtain
       	\[
       	K_i\left(C_{\mathbb{U}}^{m,n+1}\right) \cong \left(\mathbb{Z}^{2^{m + \binom{m+n}{2} - 1}}\right)^{\oplus 2^{m+n}} \cong \mathbb{Z}^{2^{m + \binom{m+n+1}{2} - 1}}, \quad i = 0,1.
       	\]
       	
       	This completes the induction and the proof.
       \end{proof}
       
\section{Twisted isometries with finite dimensional wandering spaces}

In this section, we classify the universal $C^{*}$-algebra generated by a pair $(T_{1},T_{2})$ of doubly non commuting isometries with respect to the parameter $\theta$ under the assumption that all the wandering subspaces are finite dimensional. 

\bppsn \label{KK}
Let $\xi$ and  be two full essential extensions of $\cla_{\theta}^2$ by $\mathcal{K}^{\oplus n}$ given by;
\[
\xi:0\rightarrow \mathcal{K}^{\oplus n} \rightarrow \mathbf{C} \rightarrow \mathcal{A}_{\theta}^{2}\rightarrow  0, \qquad \xi^{\prime}:0\rightarrow \mathcal{K}^{\oplus n} \rightarrow \mathbf{C}^{\prime} \rightarrow \mathcal{A}_{\theta}^{2}\rightarrow  0.
\]
If $[\xi]=[\xi^{\prime}]$ in $KK^1\left(\cla_{\theta},\, \mathcal{K}^{\oplus n} \right)$ then $\mathbf{C}$ is isomorphic to $\mathbf{C}^{\prime}$.
\eppsn 
\prf Note that $ \mathcal{A}_{\theta}^{2}$ satisfies  the Universal Coeﬃcient
Theorem.  Moreover, $\mathcal{K}^{\oplus n} \cong C(I_n)\otimes \clk$, and therefore, $Q\left(\mathcal{K}^{\oplus n}\right)$ has property $(P)$. From Theorem $2.4$ in \cite{Lin-2009aa}, it follows that $\xi$ and $\xi^{\prime}$  are unitary equivalent, hence weakly isomorphic (see page 67). Hence $C$ and $C^{\prime}$ are isomorphic. 

\qed

\begin{rmrk} \label{remark 1}
	Consider the operators $T_{1}=S^{*}\otimes 1$ and $T_{2}=e^{2\pi\mathrm{i}\theta N}\otimes S^{*}$ on $\ell^{2}(\mathbb{N})^{\otimes 2}$. Let $\mathfrak{A}_{\theta}$ denote the $C^{*}$-subalgebra of $\cll\left(\ell^2(\bbn)^{\otimes 2} \right)$ generated by $T_{1}$ and $T_{2}$. Since the representation mapping $s_1 \mapsto T_1$ and $s_2 \mapsto T_2$ is a faithful representation of $C_{\theta}^{0,2}$, it follows that  $\mathfrak{A}_{\theta}$ is isomorphic to $C_{\theta}^{0,2}$. Moreover, we have $K_{0}(\mathfrak{A}_{\theta})=\mathbb{Z}$ and $K_{1}(\mathfrak{A}_{\theta})=0$. 
\end{rmrk}

\bdfn Let $\Gamma_n=\{(\lambda_1,\cdots, \lambda_n)\in \bbbt^n: \lambda_{i}\neq\lambda_{j} \mbox{ for all } 1\leq i\neq j\leq n\}$.  Take $\Lambda \in \Gamma_n$. Consider the operators $T_{1}$ and $T_{2}$ on $\ell^2(\mathbb{N})^{\oplus n}$ defined as 
$$T_{1}=\left(S^*\right)^{\oplus n},\,\quad  T_{2}=\bigoplus_{i=1}^n \left(\lambda_i e^{-2\pi\mathrm{i}\theta N}\right).$$  
The $C^*$-algebra $\mathfrak{D}_{n,\theta}^{\Lambda}$ is defined to be the $C^{*}$-subalgebra  of $\cll\left(\ell^2(\mathbb{N})^{\oplus n}\right)$ generated by $T_{1}$ and $T_{2}$.
\edfn
It is clear that for $n=1$,  $\mathfrak{D}_{n,\theta}^{\Lambda}$  does not depend on $\Lambda$. In what follows, we will show that for higher values of $n$, $\mathfrak{D}_{n,\theta}^{\Lambda}$ is independent of $\Lambda$, and thereby, we denote it by $\mathfrak{D}_{n,\theta}$. 

\bppsn \label{short exact sequence 1} Let $\Lambda\in \Gamma_n$ and $\theta \in \bbr \setminus \bbq$.  Then 
 there is the following short exact sequence
\[
\xi^{\Lambda}_n: \quad 0\longrightarrow \mathcal{K}^{\oplus n} \xrightarrow{i}\mathfrak{D}_{n,\theta}^{\Lambda} \xrightarrow{\eta} \mathcal{A}_{\theta}^{2}\longrightarrow  0.
\]
where $\eta$ maps the generators $T_1$ and $T_2$ of $\mathfrak{D}_{n,\theta}^{\Lambda}$ to the  generators $s_1$ and $s_2$, respectively, of $\mathcal{A}_{\theta}^{2}$.
\eppsn 
\prf We have 
$$ I-T_{1} T_{1}^{*}=p^{\oplus n} \quad \mbox{ and }  \quad T_2\left(I-T_{1} T_{1}^{*}\right)= \bigoplus_{r=1}^n\left(\lambda_rp\right). $$
This shows that $0\oplus\left(\bigoplus_{j=2}^n \left(\alpha_j p\right)\right) \in \mathfrak{D}_{n,\theta}^{\Lambda}$, where $\alpha_j =1- \overline{\lambda_1}\lambda_j \notin \{0,1\}$.
Proceeding in this way, we get $0^{ \oplus (n-1)} \oplus  p \in \mathfrak{D}_{n,\theta}^{\Lambda}$. Similarly, we have $0^{\oplus (r-1)} \oplus p \oplus 0^{\oplus (n-r)}  \in \mathfrak{D}_{n,\theta}^{\Lambda}$ for all $1\leq r \leq n$.
By multiplication  with $T_{1}^{i}$ on left and ${T_{1}^{*}}^{j}$ on right in each of the above elements, we have 
$0^{\oplus (r-1)}\oplus p_{ij}\oplus 0^{\oplus (n-r)} \in \mathfrak{D}_{n,\theta}^{\Lambda}$ for all $1\leq r \leq n$. This proves that $\mathcal{K}^{\oplus n} \subset  \mathfrak{D}_{n,\theta}^{\Lambda}$ as a closed ideal.  Moreover,  the generators $[T_1]$ and $[T_2]$ of the quotient $C^*$-algebra $\mathfrak{D}_{n,\theta}^{\Lambda} / {\mathcal{K}^{\oplus n}}$ satisfy the defining relations of $\mathcal{A}_{\theta}^2$. By simplicity of $\mathcal{A}_{\theta}^2$, we get $\mathfrak{D}_{n,\theta}^{\Lambda} / \mathcal{K}^{\oplus n} \cong \mathcal{A}_{\theta}^2$. Hence we have the following short exact sequence
\[
\quad 0\longrightarrow \mathcal{K}^{\oplus n} \xrightarrow{i} \mathfrak{D}_{n,\theta}\xrightarrow{\pi} \mathcal{A}_{\theta}^{2}\longrightarrow  0.
\]
\qed 
\blmma  \label{full} 
Let $\Lambda\in \Gamma_n$ and $\theta \in \bbr \setminus \bbq$. 
Then the  short exact sequence
$\xi^{\Lambda}_n$ is a full essential extension of $\mathcal{A}_{\theta}^{2}$ by  $\mathcal{K}^{\oplus n}$.
\elmma 
\prf 
Let $\tau^{\Lambda}_n:\mathcal{A}_{\theta}^{2} \rightarrow Q\left(\mathcal{K}^{\oplus n}\right)$ be the Busby invariant corresponding to $\xi^{\Lambda}_n$.  Since $\mathcal{A}_{\theta}^{2}$  is simple, $\tau^{\Lambda}_n$ is injective, proving that   the extension $\xi^{\Lambda}_n$ is essential. To prove $\xi^{\Lambda}_n$ is full, first observe that  $Q\left(\mathcal{K}^{\oplus n}\right)\cong Q^{\oplus n} \cong C(I_n) \otimes Q$, and hence all the ideals  are of the form $\mathfrak{J}_F; \, F \subset I_n$, where $\mathfrak{J}_F=\{f:I_n \rightarrow Q: f(x)=0 \mbox{ for all } x \in F\}$.   Take a nonzero element $a \in \mathcal{A}_{\theta}^{2}$.  Then $\langle a \rangle =\mathcal{A}_{\theta}^2$ as $\mathcal{A}_{\theta}^2$ is simple, and hence we have 
\[
 \tau^{\Lambda}_n(\mathcal{A}_{\theta}^2)=  \tau^{\Lambda}_n( \langle a \rangle )\subset   \langle \tau^{\Lambda}_n(a) \rangle. 
\]
Since  $\tau^{\Lambda}_n(s_1)=\left[\left(S^*\right)^{\oplus n}\right]$, the only ideal containing  $\tau^{\Lambda}_n(a)$ is $Q\left(\mathcal{K}^{\oplus n}\right)$, proving the claim.
\qed 
\begin{ppsn}\label{General 1}
Let $\Lambda\in \Gamma_n$ and $\theta \in \bbr \setminus \bbq$.  Then one has the following: $$K_{0}\left(\mathfrak{D}_{n,\theta}^{\Lambda}\right)\cong \mathbb{Z}^{n+1} \mbox{ and } K_{1}\left(\mathfrak{D}_{n,\theta}^{\Lambda}\right)\cong\mathbb{Z}.$$
\end{ppsn}

\begin{proof}
	
	The six term short exact sequence of $K$-groups corresponding to the short exact sequence $\xi^{\Lambda}_n$ is given by the following.
	\begin{center} 
		\begin{tikzpicture}[node distance=1cm,auto]
			\tikzset{myptr/.style={decoration={markings,mark=at position 1 with %
						{\arrow[scale=2,>=stealth]{>}}},postaction={decorate}}}
			\node (A){$ K_0\left(\mathcal{K}^{\oplus n}\right)$};
			\node (Up)[node distance=2cm, right of=A][label=above:$i^{*}$]{};
			\node (B)[node distance=5cm, right of=A]{$K_0(\mathfrak{D}_{n,\theta}^{\Lambda})$};
			\node (Up)[node distance=2cm, right of=B][label=above:$\eta_*$]{};
			\node (C)[node distance=5cm, right of=B]{$K_0(\mathcal{A}_{\theta}^{2})$};
			\node (D)[node distance=2cm, below of=C]{$K_1\left(\mathcal{K}^{\oplus n}\right)$};
			\node (E)[node distance=5cm, left of=D]{$K_1(\mathfrak{D}_{n,\theta}^{\Lambda})$};
			\node (F)[node distance=5cm, left of=E]{$K_1(\mathcal{A}_{\theta}^{2})$};
			\node (Below)[node distance=2cm, left of=D][label=below:$i^{*}$]{};
			\node (Below)[node distance=2cm, left of=E][label=below:$\eta_*$]{};
			\draw[myptr](A) to (B);
			\draw[myptr](B) to (C);
			\draw[myptr](C) to node{{ $\partial$}}(D);
			\draw[myptr](D) to (E);
			\draw[myptr](E) to (F);
			\draw[myptr](F) to node{{ $\delta$}}(A);
		\end{tikzpicture}
	\end{center}
	Since the generators $s_{1},s_{2}$ of $\mathcal{A}_{\theta}^{2}$ lift to the generators $T_{1},T_{2}$ of $\mathfrak{D}_{n,\theta}^{\Lambda}$, respectively, we have $\delta\left(s_1\right)=p^{\oplus n}$, and $\delta\left(s_2\right)=0$. Using this, and the fact that $K_1\left(\mathcal{K}^{\oplus n}\right)=0$, we get the claim.
\end{proof}
\bthm   \label{homotopy} Let $n \in \bbn$ and $\Lambda, \Lambda^{\prime} \in \Gamma_n$. Then $\mathfrak{D}_{n,\theta}^{\Lambda}$ is isomorphic to $\mathfrak{D}_{n,\theta}^{\Lambda^{\prime}}$.
\ethm 
\prf  Since $\Gamma_n$ is path-connected, one gets a path $f(t); \, 0\leq t \leq 1,$ joining $\Lambda$ and  $\Lambda^{\prime}$. For $t \in [0,1]$, define a homomorphism $G_t:\mathcal{A}_{\theta}^{2} \rightarrow Q\left(\mathcal{K}^{\oplus n}\right)$ by 
$$ G_t(s_1)=s_1, \mbox{ and } G_t(s_2)=f(t)s_2. $$  Consider the map $G:[0,1]\times \mathcal{A}_{\theta}^{2} \rightarrow Q\left(\mathcal{K}^{\oplus n}\right)$ given by $G(t,s_1)=G_t(a)$ for $a \in \mathcal{A}_{\theta}^{2}$.  The map $G$ constitutes a homotopy between $\xi^{\Lambda}_n$ and $\xi^{\Lambda^{\prime}}_n$, thus, 
$[\xi^{\Lambda}_n]=[\xi^{\Lambda^{\prime}}_n]$ in $KK^1\left(\cla_{\theta},\,\mathcal{K}^{\oplus n} \right)$. Hence the claim follows by Propositions (\ref{full}, \ref{KK}).
\qed 

\bdfn \label{Fmn}  Let $\Lambda^1\in \Gamma_m$, and $ \Lambda^2\in \Gamma_n$. Consider the operators $T_{1}$ and $T_{2}$ on $\ell^2(\mathbb{N})^{\oplus (m+n)}$ defined as follows: 
$$T_{1}= \left(S^*\right)^{\oplus m}\oplus \left(\bigoplus_{j=1}^n \left(\lambda_j e^{2\pi\mathrm{i}\theta N}\right)\right),\quad  T_{2}= \bigoplus_{i=1}^m \left(\lambda_i^{'} e^{2\pi\mathrm{i}\theta N}\right)\oplus \left(S^*\right)^{\oplus n}.$$
The $C^*$-algebra $\mathfrak{F}_{m,n,\theta}^{\Lambda^1,\Lambda^2}$ is defined to be the $C^{*}$-subalgebra  of $\cll\left(\ell^2(\mathbb{N})^{\oplus (m+n)}\right)$ generated by $T_{1}$ and $T_{2}$.
\edfn 

\bppsn \label{short exact sequence 2} Then 
there is the following short exact sequence
\[
\chi^{\Lambda^1,\Lambda^{2}}_{m,n}: \quad 0\longrightarrow \mathcal{K}^{\oplus (m+n)} \xrightarrow{i} \mathfrak{F}_{m,n,\theta}^{\Lambda^1,\Lambda^2} \xrightarrow{\zeta} \mathcal{A}_{\theta}^{2}\longrightarrow  0.
\]
where $\zeta$ maps the generators $T_1$ and $T_2$ of $\mathfrak{F}_{m,n,\theta}^{\Lambda^1,\Lambda^2}$ to the  generators $s_1$ and $s_2$, respectively, of $\mathcal{A}_{\theta}^{2}$. Moreover, $\chi^{\Lambda^1,\Lambda^{2}}_{m,n}$ is a full essential extension of $\mathcal{A}_{\theta}^{2}$ by  $\mathcal{K}^{\oplus (m+n)}$.
\eppsn 
\prf 
We have 
$I- T_{1} T_{1}^{*}= p^{\oplus m}\oplus 0^{\oplus n}$ and $I-T_{2} T_{2}^{*}= 0^{\oplus m}\oplus p^{\oplus n}$. Also, 
$$T_2\left(I- T_{1}T_{1}^{*}\right)= \oplus_{i=1}^m \left(\lambda_i^{'} p\right)\oplus 0^{\oplus n} \mbox{ and }  \,T_1\left(I-T_{2}T_{2}^{*}\right)= 0^{\oplus m}\oplus\left(\oplus_{j=1}^n \left(\lambda_j p\right)\right).$$
Using this and proceeding along the lines of the proofs of Proposition \ref{short exact sequence 1} and Lemma \ref{full}, we obtain the claim.
\qed 
\begin{ppsn}\label{General 2}
	In the above setup, $K_{0}\left(\mathfrak{F}_{m,n,\theta}^{\Lambda^1,\Lambda^2}\right)\cong \mathbb{Z}^{m+n}$ and $K_{1}\left((\mathfrak{F}_{m,n,\theta}^{\Lambda^1,\Lambda^2}\right)\cong 0.$ 
\end{ppsn}

\begin{proof}
		The six term short exact sequence of $K$-groups corresponding to the short exact sequence $\chi^{\Lambda^1,\Lambda^{2}}_{m,n}$ is given by the following.
	\begin{center} 
		\begin{tikzpicture}[node distance=1cm,auto]
			\tikzset{myptr/.style={decoration={markings,mark=at position 1 with %
						{\arrow[scale=2,>=stealth]{>}}},postaction={decorate}}}
			\node (A){$ K_0\left(\mathcal{K}^{\oplus (m+n)}\right)$};
			\node (Up)[node distance=2cm, right of=A][label=above:$i^{*}$]{};
			\node (B)[node distance=5cm, right of=A]{$K_0\left(\mathfrak{F}_{m,n,\theta}^{\Lambda_1,\Lambda_2}\right)$};
			\node (Up)[node distance=2cm, right of=B][label=above:$\zeta_*$]{};
			\node (C)[node distance=5cm, right of=B]{$K_0\left(\mathcal{A}_{\theta}^{2}\right)$};
			\node (D)[node distance=2cm, below of=C]{$K_1\left(\mathcal{K}^{\oplus (m+n)}\right)$};
			\node (E)[node distance=5cm, left of=D]{$K_1\left(\mathfrak{F}_{m,n,\theta}^{\Lambda_1,\Lambda_2}\right)$};
			\node (F)[node distance=5cm, left of=E]{$K_1\left(\mathcal{A}_{\theta}^{2}\right)$};
			\node (Below)[node distance=2cm, left of=D][label=below:$i^{*}$]{};
			\node (Below)[node distance=2cm, left of=E][label=below:$\zeta_*$]{};
			\draw[myptr](A) to (B);
			\draw[myptr](B) to (C);
			\draw[myptr](C) to node{{ $\partial$}}(D);
			\draw[myptr](D) to (E);
			\draw[myptr](E) to (F);
			\draw[myptr](F) to node{{ $\delta$}}(A);
		\end{tikzpicture}
	\end{center}
	Since the generators $s_{1},s_{2}$ of $\mathcal{A}_{\theta}^{2}$ lift to the generators $T_{1},T_{2}$ of $\mathfrak{F}_{m,n,\theta}^{\Lambda_1,\Lambda_2}$, respectively, we have 
$$\delta(s_1)=p^{\oplus m}\oplus 0^{\oplus n}, \quad  \mbox{ and } \quad \delta(s_2)= 0^{\oplus m} \oplus p^{\oplus n}.$$ 
Using this, and the fact that $K_1\left(\mathcal{K}^{\oplus (m+ n}\right)=0$, we get the claim.
\end{proof}

\bthm \label{homotopy1}  Let $m, n \in \bbn$ and $(\Lambda^1,\Lambda^2),  (\gamma^1,\gamma^2) \in \Gamma_m \times \Gamma_n$. 
Then $\mathfrak{F}_{m,n,\theta}^{\Lambda^1,\Lambda^2}$ is isomorphic to $\mathfrak{F}_{m,n,\theta}^{\gamma^1,\gamma^2}$.
\ethm 
\prf The proof follows along similar lines as the proof of Theorem \ref{homotopy}.
\qed

\brmrk From now on, we denote $\mathfrak{F}_{m,n,\theta}^{\gamma^1,\gamma^2}$ simply  by $\mathfrak{F}_{m,n,\theta}$, without any ambiguity, thanks to Theorem \ref{homotopy}. 
\ermrk 

Let $\mathbf{T}=(T_{1},T_{2})$ be a pair of doubly non-commuting isometries on a Hilbert space $\mathcal{H}$ with finite-dimensional wandering spaces. Since $T_{2}$ is an isometry and $\mathcal{W}_{\{1\}}$ is finite dimensional, $T_{2}|_{\mathcal{W}_{\{1\}}}$ is a unitary. Thus $T_{2}^{m_{2}}\mathcal{W}_{\{1\}}=\mathcal{W}_{\{1\}}$. Similarly, we  have $T_{1}^{m_{1}}\mathcal{W}_{\{2\}}=\mathcal{W}_{\{2\}}$. Hence, in the  von Neumann-Wold decomposition of $\mathbf{T}$,  we get  $$\mathcal{H}=\oplus_{A\subseteq\{1,2\}}\mathcal{H}_{A},$$ where
\begin{enumerate}[(i)]
	\item $\mathcal{H}_{\phi}=\cap_{m_{1},m_{2}}T_{1}^{m_{1}}T_{2}^{m_{2}}\mathcal{H}$
	\item $\mathcal{H}_{\{1\}}=\oplus_{m_{1}}T_{1}^{m_{1}}(\cap_{m_{2}}T_{2}^{m_{2}}\textrm{ker}T_{1}^{*})
	=\oplus_{m_{1}}T_{1}^{m_{1}}(\cap_{m_{2}}T_{2}^{m_{2}}\mathcal{W}_{\{1\}})	=\oplus_{m_{1}}T_{1}^{m_{1}}\mathcal{W}_{\{1\}}$
	\item $\mathcal{H}_{\{2\}}=\oplus_{m_{2}}T_{2}^{m_{2}}(\cap_{m_{1}}T_{1}^{m_{1}}\textrm{ker}T_{2}^{*})=\oplus_{m_{2}}T_{2}^{m_{2}}(\cap_{m_{1}}T_{1}^{m_{1}}\mathcal{W}_{\{2\}})
	=\oplus_{m_{2}}T_{2}^{m_{2}}\mathcal{W}_{\{2\}}$
	\item $\mathcal{H}_{\{1,2\}}=\oplus_{m_{1},m_{2}}T_{1}^{m_{1}}T_{2}^{m_{2}}(\textrm{ker}T_{1}^{*}\cap\textrm{ker}T_{2}^{*})=\oplus_{m_{1},m_{2}}T_{1}^{m_{1}}T_{2}^{m_{2}}\mathcal{W}_{\{1,2\}}$
\end{enumerate}
The following theorem gives a classification of the algebra $C^{*}(\mathbf{T})$ for a tuple $\mathbf{T}$ of doubly non commuting isometries.

\begin{thm}\label{Classification}
	Let $\mathbf{T}=(T_{1},T_{2})$ be a pair of  doubly-twisted   isometries with  parameter $\theta\in\mathbb{R}\setminus \bbq$ acting  on a Hilbert space $\clh$. Assume that all the wandering subspaces are finite dimensional. Then we have the following cases.
	\begin{enumerate}[(i)]
		\item If $\mathcal{H}_{\{1,2\}}\neq0$, then $C^{*}(\mathbf{T})\cong C_{\theta}^{0,2}.$
		\item If $\mathcal{H}_{\{1,2\}}=\mathcal{H}_{\{2\}}=0$ and $\mathcal{H}_{\{1\}}\neq 0$ then $C^{*}(\mathbf{T})\cong \mathfrak{D}_{n,\theta}$, where  $n$ is the number of distinct eigenvalues of $\restr{T_2}{\mathcal{W}_{\{1\}}}$.
		\item If $\mathcal{H}_{\{1,2\}}=\mathcal{H}_{\{1\}}=0,\mathcal{H}_{\{2\}}\neq 0$, then $C^{*}(\mathbf{T})\cong \mathfrak{D}_{n,\theta}$, where  $n$ is the number of distinct eigenvalues of $\restr{T_1}{\mathcal{W}_{\{2\}}}.  $
		\item If $\mathcal{H}_{\{1,2\}}=0$ and $\mathcal{H}_{\{1\}},\mathcal{H}_{\{2\}}\neq 0$, then $C^{*}(\mathbf{T})\cong \mathfrak{F}_{m,n,\theta}$, where  $m$ and $n$ are the number of distinct eigenvalues of $\restr{T_2}{\mathcal{W}_{\{1\}}}$ and $\restr{T_1}{\mathcal{W}_{\{2\}}}$, respectively.
		\item Suppose $\mathcal{H}_{\{1,2\}}=\mathcal{H}_{\{1\}}=\mathcal{H}_{\{2\}}=0$ and  $\mathcal{H}_{\phi}\neq 0$. Then $C^{*}(\mathbf{T})\cong \mathcal{A}^2_{\theta}$.
	\end{enumerate}	
\end{thm}
\begin{proof}
We split the proof into different cases.
	\begin{enumerate}[(i)]
		\item Let $\mathcal{H}_{\{1,2\}}\neq0$. Let $\{l_{k}\}_{k=1}^{n}$ be an orthonormal basis for $\mathcal{W}_{\{1,2\}}$. Define the following subspace for $1\leq k\leq n$. $$\mathcal{H}_{\{1,2\}}^{k}=\oplus_{n,m}T_{1}^{n}T_{2}^{m}(l_{k}).$$ Thus $\mathcal{H}_{\{1,2\}}=\oplus_{k}\mathcal{H}_{\{1,2\}}^{k}$. Moreover,  the set $\{T_{1}^{m_{1}}T_{2}^{m_{2}}(l_{k}):m_{1},m_{2}\in\mathbb{N}\}$ is an orthonormal basis for $\mathcal{H}_{\{1,2\}}^{k}$. Consider the following map.
		$$\psi:\mathcal{H}_{\{1,2\}}^{k}\rightarrow \ell^{2}(\mathbb{N})\otimes\ell^{2}(\mathbb{N})\,\,;\,\,T_{1}^{n}T_{2}^{m}(l_{k})\mapsto e_{n}\otimes e_{m}.$$
		Hence  $\mathcal{H}_{\{1,2\}}^{k}\cong \ell^{2}(\mathbb{N})\otimes\ell^{2}(\mathbb{N})$ for every $1\leq k\leq n$. With this identification, we have 
		\begin{eqnarray*}
			T_{1}(e_{n}\otimes e_{m})
			&=&T_{1}(T_{1}^{n}T_{2}^{m}(l_{k}))\\
			&=&T_{1}^{n+1}T_{2}^{m}(l_{k})\\
			&=&(e_{n+1}\otimes e_{m})\\
			&=&(S^{*}\otimes 1)(e_{n}\otimes e_{m})
		\end{eqnarray*}
		Similarly, one has 
		\begin{eqnarray*}
			T_{2}(e_{n}\otimes e_{m})
			&=&T_{2}(T_{1}^{n}T_{2}^{m}(l_{k}))\\
			&=&e^{-2\pi\mathrm{i}\theta n}T_{1}^{n}T_{2}^{m+1}(l_{k})\\
			&=&e^{-2\pi\mathrm{i}\theta n}e_{n}\otimes e_{m+1}\\
			&=&(e^{-2\pi\mathrm{i}\theta N}\otimes S^{*})(e_{n}\otimes e_{m})
		\end{eqnarray*}
	This implies that
		\[
		T_{1}|_{\mathcal{H}_{\{1,2\}}}= \left(S^{*}\otimes 1\right)^{\oplus n},  \quad  \mbox{ and }  \quad 	T_{2}|_{\mathcal{H}_{\{1,2\}}}= 	\oplus_{k=1}^n \left(e^{-2\pi\mathrm{i}\theta N}\otimes S^{*}\right)^{\oplus n}.\]
Hence we get
		$$
		T_{1}=	(S^{*} \otimes 1)	 \oplus T_{1}^{\{1\}}  \oplus T_{1}^{\{2\}} \oplus T_{1}^{\phi}  \mbox{ and } 
	T_{2}=	(e^{-2\pi\mathrm{i}\theta N}\otimes S^{*}) \oplus T_{2}^{\{1\}}  \oplus T_{2}^{\{2\}} \oplus T_{2}^{\mathcal{H}_{\phi}}
		$$
		For $A \subset \{1,2\}$, define $C^A=C^*(\{T_1^A,T_2^A\})$.  By the defining relations of $C^{0,2}_{\theta}$, one has a homomorphism $\phi_A: C^{0,2}_{\theta} \rightarrow C^A$ mapping $s_1 \mapsto  T_1^A, \,\,  s_2 \mapsto T_2^A$. Moreover, for $A=\{1,2\}$, the map $\phi_A$ is an isomorphism. 
		By Proposition \ref{isomdirectsum}, the homomorphism 
		\begin{eqnarray*} 
			\psi &:C^{0,2}_{\theta} &\longrightarrow C^{*}(\mathbf{T})\\
			& & s_1 \mapsto T_{1}= (S^{*} \otimes 1)	 \oplus T_{1}^{\{1\}}  \oplus T_{1}^{\{2\}} \oplus T_{1}^{\phi}\\
			& & s_2 \mapsto T_{2}=(e^{-2\pi\mathrm{i}\theta N}\otimes S^{*}) \oplus T_{2}^{\{1\}}  \oplus T_{2}^{\{2\}} \oplus T_{2}^{\mathcal{H}_{\phi}}
		\end{eqnarray*} 
		is an isomorphism. 
		
		\item Let $\alpha_1, \cdots \alpha_n$ be eigenvalues of  $\restr{T_2}{\mathcal{W}_{\{1\}}}$, and $\{l_{r_j}\}_{j=1}^{n_r}$ be an orthonormal basis for eigenspace corresponding to $\alpha_r$. Define
		 $$\mathcal{H}_{\{1\}}^{r,j}=\oplus_{m\in\mathbb{N}}T_{1}^{m}(l_{r_j}) \quad  \mbox{ and }\quad  \mathcal{H}_{\{1\}}^{r}=\oplus_{j=1}^{n_r}\mathcal{H}_{\{1\}}^{r,j}.$$ 
		 Then we have  $\mathcal{H}_{\{1\}}=\oplus_{r=1}^n\mathcal{H}_{\{1\}}^{r}.$  Identifying $T_{1}^{m}(l_{r_j})$ with $e_m$, one can view $\mathcal{H}_{\{1\}}^{r,j}$ as $\ell^2(\bbn)$, and thus 
		$$T_{1}|_{\mathcal{H}_{\{1\}}^{r,j}}=S^{*} \mbox{ and } T_{2}|_{\mathcal{H}_{\{1\}}^{r,j}}=\alpha_r e^{-2\pi\mathrm{i}\theta N}.$$ 
		This implies that 
		\[
		T_{1}|_{\mathcal{H}_{\{1\}}}= \oplus_{r=1}^n\, \oplus_{j=1}^{n_r} S^* \quad \mbox{ and } \quad 	T_{2}|_{\mathcal{H}_{\{1\}}}= \oplus_{r=1}^n\, \oplus_{j=1}^{n_r} \alpha_r e^{-2\pi\mathrm{i}\theta N}.
		\]
		Hence we have 
		\[
		T_1=	T_{1}|_{\mathcal{H}_{\{1\}}}\oplus T_1^{\phi}=\big( \oplus_{r=1}^n\, \oplus_{j=1}^{n_r} S^* \big) \oplus T_1^{\phi} \mbox{ and } T_2=	T_{2}|_{\mathcal{H}_{\{1\}}}\oplus T_2^{\phi}= \big(\oplus_{r=1}^n\, \oplus_{j=1}^{n_r} \alpha_r e^{-2\pi\mathrm{i}\theta N}\big) \oplus T_2^{\phi}.
		\]
		By Proposition (\ref{isomdirectsum1}, \ref{isomdirectsum2}), the $C^*$-algebra  $C^*(\mathbf{T})$ is isomorphic to the $C^*(\{R_1,R_2\})$, where 
		$$R_1=\oplus_{r=1}^n\ S^*  \mbox{ and } R_2= \oplus_{r=1}^n\, \alpha_r e^{-2\pi\mathrm{i}\theta N}.$$ 
		This proves the claim. 
		\item The proof of this case is exactly along the lines of part $(ii)$.
			\item Let $\alpha_1, \cdots \alpha_m$ be eigenvalues of  $\restr{T_2}{\mathcal{W}_{\{1\}}}$, and $\{l_{r_j}\}_{j=1}^{m_r}$ be an orthonormal basis for eigenspace corresponding to $\alpha_r$. Let $\beta_1, \cdots \beta_n$ be eigenvalues of  $\restr{T_1}{\mathcal{W}_{\{2\}}}$, and $\{k_{r_j}\}_{j=1}^{n_r}$ be an orthonormal basis for eigenspace corresponding to $\beta_r$. Define
		$$\mathcal{H}_{\{1\}}^{r,j}=\oplus_{m\in\mathbb{N}}T_{1}^{m}(l_{r_j}) \quad  \mbox{ and }\quad  \mathcal{H}_{\{1\}}^{r}=\oplus_{j=1}^{n_r}\mathcal{H}_{\{1\}}^{r,j}.$$ 
		Then we have  $\mathcal{H}_{\{1\}}=\oplus_{r=1}^n\mathcal{H}_{\{1\}}^{r}.$  Identifying $T_{1}^{m}(l_{r_j})$ with $e_m$, one can view $\mathcal{H}_{\{1\}}^{r,j}$ as $\ell^2(\bbn)$, and thus 
		$$T_{1}|_{(\mathcal{H}_{\{1\}}^{r,j})}=S^{*} \mbox{ and } T_{2}|_{(\mathcal{H}_{\{1\}}^{r,j})}=\alpha_r e^{-2\pi\mathrm{i}\theta N}.$$ 
		This implies that 
		\[
		T_{1}|_{\mathcal{H}_{\{1\}}}= \oplus_{r=1}^n\, \oplus_{j=1}^{n_r} S^* \quad \mbox{ and } \quad 	T_{2}|_{\mathcal{H}_{\{1\}}}= \oplus_{r=1}^n\, \oplus_{j=1}^{n_r} \alpha_r e^{-2\pi\mathrm{i}\theta N}.
		\]
		Hence we have 
		\[
		T_1=	T_{1}|_{\mathcal{H}_{\{1\}}}\oplus T_1^{\phi}=\big( \oplus_{r=1}^n\, \oplus_{j=1}^{n_r} S^* \big) \oplus T_1^{\phi} \mbox{ and } T_2=	T_{2}|_{\mathcal{H}_{\{1\}}}\oplus T_2^{\phi}= \big(\oplus_{r=1}^n\, \oplus_{j=1}^{n_r} \alpha_r e^{-2\pi\mathrm{i}\theta N}\big) \oplus T_2^{\phi}.
		\]
		By Proposition (\ref{isomdirectsum1}, \ref{isomdirectsum2}), the $C^*$-algebra  $C^*(\mathbf{T})$ is isomorphic to the $C^*(\{R_1,R_2\})$, where 
		$$R_1=\oplus_{r=1}^n\ S^*  \mbox{ and } R_2= \oplus_{r=1}^n\, \alpha_r e^{-2\pi\mathrm{i}\theta N}.$$ 
		This proves the claim. 
		\item This follows from the fact that $T_1$ and $T_2$ satisfy the defining relations of $\mathcal{A}^2_{\theta}$, and that $\mathcal{A}^2_{\theta}$ is simple.
	\end{enumerate}	
\end{proof}   
The following theorem says that $K$-theory is a complete invariant for 	a $C^*$-algebra generated by  a pair of  doubly-twisted   isometries  with finite wandering spaces.
\bthm Let $\mathbf{T}=(T_{1},T_{2})$ and   $\mathbf{T}^{\prime}=(T_{1}^{\prime},T_{2}^{\prime})$ be two pairs of  doubly-twisted   isometries with  parameter $\theta\in\mathbb{R}\setminus \bbq$ acting  on a Hilbert space $\clh$. Assuming both $\mathbf{T}$ and $\mathbf{T}^{\prime}$ have finite-dimensional wandering spaces. Then 
$$C^*(\mathbf{T})\cong C^*(\mathbf{T}^{\prime}) \mbox{ if and only if } K_*(C^*(\mathbf{T}))\cong K_*(C^*(\mathbf{T}^{\prime})).$$
\ethm 
\prf This is a straightforward consequence of Theorem \ref{Classification}, Propositions (\ref{General 1}, \ref{General 2}), and Theorem $6.8$ in \cite{Web-2013aa}.
\qed 
\bdfn Let \( \Delta_{\theta} \) be the set of \( C^* \)-algebras generated by doubly twisted isometries \( (T_1, T_2) \leadsto C^{0,2}_{\theta} \), with finite-dimensional wandering spaces, considered up to isomorphism.  
We call a \( C^* \)-algebra \( C \) of \textbf{type I}, \textbf{type II}, \textbf{type III}, or \textbf{type IV} if it is isomorphic to \( C^{0,2}_{\theta} \), \( \mathfrak{D}_{n,\theta} \), \( \mathfrak{F}_{m,n,\theta} \), or \( \mathcal{A}_{\theta}^2 \), respectively.
\edfn 
Let $\mathbf{T}=(T_{1},T_{2})$ be a pair of doubly twisted isometries with respect to a unitary operator $U$ on a Hilbert space $\mathcal{H}$. In the Wold decomposition of $\mathbf{T}$, assume that for every  $A\subseteq\{1,2\}$, the wandering subspace $\mathcal{W}_{A}$ is finite dimensional. 

\begin{lmma}
Under the above set-up,	the spectrum $\sigma(U)$ is a finite set.	
\end{lmma}

\begin{proof}
	For every $A\subseteq\{1,2\}$, $U|_{\mathcal{W}_{A}}$ is diagonalisable and the spectrum $\sigma(U|_{\mathcal{W}_{A}})$ is finite. 
	By definition of $\mathcal{H}_{A}$, and the map $$ \phi:\mathcal{H}_{A}\longrightarrow \ell^{2}(\mathbb{N})^{\otimes\,|A^{c}|}\otimes\mathcal{W}_{\{1,2\}}\,;\,\overrightarrow{\prod_{i\in A^{c}}}T_{i}^{m_{i}}w \longrightarrow \overrightarrow{\otimes_{i\in A^{c}}}e_{m_{i}}\otimes w.$$ 
	Since $U$ commutes with $T_{1}$ and $T_{2}$,  one has $$U|_{\mathcal{H}_{A}}=U(\overrightarrow{\oplus_{i\in A^{c}}}T_{i}^{m_{i}}\mathcal{W}_{A})=\overrightarrow{\oplus_{i\in A^{c}}}T_{i}^{m_{i}}U(\mathcal{W}_{A}).$$ Therefore we have
	$$U|_{\mathcal{H}_{A}}\cong 1^{\otimes\,|A^{c}|} \otimes U|_{\mathcal{W}_{A}}.$$ This implies that  $\sigma(U|_{\mathcal{H}_{A}})=\sigma(U|_{\mathcal{W}_{A}})$. Therefore $$\sigma(U)=\sigma(\oplus_{A\subseteq\{1,2\}}U^A)=\cup_{A\subseteq\{1,2\}}\sigma(U|_{\mathcal{H}_{A}})=\cup_{A\subseteq\{1,2\}}\sigma(U|_{\mathcal{W}_{A}}).$$ 
	This proves the claim.
\end{proof}

\begin{thm} \label{U-twisted}
	Let $U$ be a unitary operator acting on a Hilbert space $\mathcal{H}$ having finite spectrum $\{e^{2\pi\mathrm{i}\theta_{1}}, \cdots, e^{2\pi\mathrm{i}\theta_{k}}\}$, where $\theta_i \in \bbr \setminus\bbq$ for $1\leq i \leq k$.	Let $\mathbf{T}=(T_{1},T_{2})$ be a pair  of isometries  on  $\mathcal{H}$ such that $\mathbf{T} \leadsto C^{0,2}_U$. Assume that all the wandering subspaces in the Wold decomposition of $\mathbf{T}$ are finite dimensional. Then  $$C^{*}(\mathbf{T}) \cong \oplus_{i=1}^{k}C_{\theta_i},$$ 
	where   $C_{\theta_i} \in \Delta_{\theta_i}$ for all $1 \leq i \leq k$.
\end{thm}	  

\prf 
For each $1\leq i\leq k$, the functions $\chi_{e^{2\pi\mathrm{i}\theta_{i}}}\in C(\sigma(U))$ maps to  orthogonal projections $P_{i}\in C^{*}(\{U\})$ under continuous functional calculus  such that 
$$\sum_{i=1}^{k}P_{i}=1 \mbox{ and } \sum_{i=1}^k\, e^{2\pi\mathrm{i}\theta_{i}}P_i=U.$$  
Moreover, since $U \in \mathcal{Z}(C^*(\mathbf{T}))$, it follows that $P_i \in \mathcal{Z}(C^*(\mathbf{T}))$, and  we have  
 $$a=\sum_{i=1}^{k}P_{i}aP_{i}, \quad  \mbox{ and }  \quad (P_iaP_i)(P_jaP_j)=0 \mbox{ if } i \neq j, $$ 
for all  $a \in  C^*(\mathbf{T})$.
Hence
$$C^*(\mathbf{T})=
C^{*}(\{T_{1},T_{2}\})=C^{*}\Big(\big\{\sum_{i=1}^{k}P_iT_1P_i,\sum_{i=1}^{k}P_iT_2P_i\big\}\Big)=
\oplus_{i=1}^{k}C^{*}(P_{i}T_{1}P_i,P_{i}T_{2}P_i).$$
Note that $$(P_iT_1P_i)^*(P_iT_2P_i)=P_iT_1^*T_2P_i=P_iU^*T_2T_1^*P_i %=(P_kU^*P_k)(P_kT_2P_k)(P_kT_1^*P_k)
=e^{2\pi\mathrm{i}\theta_{i}}(P_iT_2P_i)(P_iT_1^*P_i).$$
Define $\mathbf{T}^i=(P_{j}T_{1}P_i,P_{j}T_{2}P_i)$ for $1\leq i \leq k$. Then $\mathbf{T}^i \leadsto C^{0,2}_{\theta_i}$.  Moreover, all the wandering subspaces of each $\mathbf{T}^i$ are finite-dimensional, since the same holds for $\mathbf{T}$.
Define $C_{\theta_i}=C^{*}(P_{j}T_{1}P_i,P_{j}T_{2}P_i)$ for $1\leq i \leq k$. Then $C_{\theta_i} \in  \Delta_{\theta_i}$ and 
 $C^{*}(\mathbf{T}) \cong \oplus_{i=1}^{k}C_{\theta_i},$  proving  the claim.  
\qed 
\brmrk We call $C_{\theta}$'s the $\theta$-component  of $C^*(\mathbf{T})$ with respect to $U$.  Observe  that $\theta$-component  and  $(\theta+n)$-component are the same for any $n \in \bbz$. 
\ermrk
\bppsn \label{maximal ideal}
Let \( J_{n}^{\mu} = \langle 1 - T_1T_1^* \rangle \). 
Then \( J_{n}^{\mu}= \mathcal{K}^{\oplus n} \). 
If \( \mu \) is irrational, then \( J_{n}^{\mu}\) is the unique maximal ideal of \( \mathfrak{D}_{n,\mu} \).
\eppsn

\prf That  \( J_{n}^{\mu}= \mathcal{K}^{\oplus n}\) follows from the calculations done in the proof of Proposition \ref{short exact sequence 1}. To prove the rest, consider the surjective $*$-homomorphism $\varPsi: C^{0,2}_{\theta} \rightarrow  \mathfrak{D}_{n,\mu}$ sending $s_1 \mapsto T_1$ and $s_2 \mapsto T_2$.  Let $J=\langle 1-s_1s_1^*, 1-s_2s_2^*\rangle$. Then, by Proposition $3.2$ in \cite{Web-2013aa}, it follows that $J$ is the unique maximal ideal of  $C^{0,2}_{\theta}$, and hence $\ker \varPsi\subset J$.  Since $\varPsi(J)=J_{n}^{\mu}$, the claim follows.  
\qed 
\bppsn \label{maximal ideal1}
Let \( J_{m,n}^{\mu} = \langle 1 - T_1T_1^*,\ 1 - T_2T_2^* \rangle \). 
Then \( J_{m,n}^{\mu} = \mathcal{K}^{\oplus (m+n)} \). 
If \( \theta \) is irrational, then \( J_{m,n}^{\mu} \) is the unique maximal ideal of \( \mathfrak{F}_{m,n,\mu} \).
\eppsn
\prf The proof is similar to that of Proposition \ref{maximal ideal}, and is therefore omitted.
\qed

\bppsn \label{plusminus}
 Let $\mu, \nu \in \bbr \setminus \bbq$. Then one has the following.
\begin{enumerate}[(i)]
	\item $\mathfrak{D}_{n,\mu}\cong \mathfrak{D}_{n,\nu} \iff  \mu =\pm \nu \mbox{ mod } \bbz. $
	\item $\mathfrak{F}_{m,n,\mu}\cong\mathfrak{F}_{m,n,\nu} \iff  \mu =\pm \nu \mbox{ mod } \bbz. $
\end{enumerate}
\eppsn 
\prf
To prove the forward implication of part (i), let  
\[
\varpi: \mathfrak{D}_{n,\mu} \rightarrow \mathfrak{D}_{n,\nu}
\]  
be an isomorphism. Then, by Proposition \ref{maximal ideal}, we have \( \varpi(J_{n}^{\mu}) = J_{n}^{\nu} \). This induces an isomorphism between  
\[
\mathfrak{D}_{n,\mu} / J_{n}^{\mu} \cong \mathcal{A}_{\mu}^2 \quad \text{and} \quad \mathfrak{D}_{n,\nu} / J_{n}^{\nu} \cong \mathcal{A}_{\nu}^2,
\]  
and hence we conclude that \( \mu = \pm \nu \mod \mathbb{Z} \).
For the converse, note that  
\[
\mathfrak{D}_{n,\mu} = C^*(\{T_1, T_2\}) = C^*(\{T_1, T_2^*\}) = \mathfrak{D}_{n,-\mu}.
\]

For part (ii), we will prove the necessary condition, as the sufficient one follows as in the previous case.  
For that, let \( \mathbf{T}^{+} = (T_1^+, T_2^+) \) and \( \mathbf{T}^{-} = (T_1^-, T_2^-) \) denote the corresponding generators as given in Definition \ref{Fmn}. It is straightforward to verify that all wandering data of \( \mathbf{T}^{+} \) and \( \mathbf{T}^{-} \) are the same. Hence, by Theorem 5.2 in \cite{RakSarSur-2022aa}, the tuples \( \mathbf{T}^{+} = (T_1^+, T_2^+) \) and \( \mathbf{T}^{-} = (T_1^-, T_2^-) \) are unitarily equivalent, and thus  
\[
\mathfrak{F}_{m,n,\mu} = C^*(\mathbf{T}^{+}) \cong C^*(\mathbf{T}^{-}) \cong \mathfrak{F}_{m,n,-\mu}.
\]
\qed
 \brmrk \label{plusminus1}
By Proposition 3.6 in \cite{Web-2013aa}, the above result holds for \( C^{0,2}_{\theta} \). The same result holds for \( \mathcal{A}_{\theta}^2 \).  
\ermrk

\blmma \label{plusminus2}
  Let $\mu, \nu \in \bbr \setminus \bbq$. Let $B_{\mu}$ and $C_{\nu}$ be two $C^*$-algebras such that  $B_{\mu} \in  \Delta_{\mu}$ and $C_{\nu} \in  \Delta_{\nu}$. If 
$ B_{\mu} \cong C_{\nu} $, then $ \mu =\pm \nu \mbox{ mod } \bbz. $ 
\elmma 
\prf By comparing the \( K \)-groups, one concludes that both \( B_{\mu} \) and \( C_{\nu} \) are of the same type. The claim then follows directly from Proposition \ref{plusminus} and Remark \ref{plusminus1}.  
\qed

		\bthm \label{U-twisted 1}
	Let $U$ be a unitary operator acting on a Hilbert space $\mathcal{H}$ with finite spectrum $\{e^{2\pi\mathrm{i}\theta_{1}}, \ldots, e^{2\pi\mathrm{i}\theta_{k}}\}$, where $\theta_i \in \bbr \setminus \bbq$ for $1 \leq i \leq k$. Let $\mathbf{T} = (T_{1}, T_{2})$ and $\mathbf{T}^{\prime} = (T_{1}^{\prime}, T_{2}^{\prime})$ be two pairs of isometries on $\mathcal{H}$ such that
$	\mathbf{T}, \mathbf{T}^{\prime} \leadsto ^*C^{0,2}_U $. 
	Assume that for \( 1 \leq i \leq k \), the algebras \( C_{\theta_i} \) and \( C'_{\theta_i} \) are the \( \theta_i \)-components of \( C^*(T) \) and \( C^*(T') \), respectively, with respect to \( U \). Then the following holds:
	\begin{enumerate}[(i)]
		\item If \( e^{2\pi i \theta_i} \in \sigma(U) \) but \( e^{-2\pi i \theta_i} \notin \sigma(U) \), then \( C_{\theta_i} \cong C'_{\theta_i} \).
		\item If both \( e^{2\pi i \theta_i}, e^{-2\pi i \theta_i} \in \sigma(U) \), then either \( C_{\theta_i} \cong C'_{\theta_i} \) or \( C_{\theta_i} \cong C'_{-\theta_i} \).
	\end{enumerate}
	\ethm 
\prf It suffices to prove the forward implication, thanks to Theorem \ref{U-twisted}. Let 
$$ \phi: C^*(\mathbf{T}) \rightarrow   C^*(\mathbf{T}^{\prime}) $$
be an isomorphism. Let $\phi(U) = U^{0}$. Then we have 
$$\sigma(U)=\sigma(U^0) \mbox{  and } UU^{0}=U^{0}U \mbox{  as } U, U^{0}\in \mathcal{Z}(C^*(\mathbf{T}^{\prime})).$$ 
Observe that  $\mathbf{T}^{\prime}$ is a $U^{0}$-doubly twisted tuple of isometries with finite dimensional wandering spaces. Let $C^0_{\theta_i}$ be the $\theta_i$-component of $C^*(\mathbf{T}^{\prime})$ with respect to $U^0$.  Let $P_i$ and $P_i^0$ be the orthogonal projection onto the eigenspace of $U$ and $U^0$, respectively, corresponding to eigenvalue $e^{2\pi i \theta_i}$. Then $\phi(P_i) = P_i^0$ for all $1 \leq i \leq k$. Hence we have 
$$ C^0_{\theta_i}= P_i^0C^*(\mathbf{T}^{\prime})P_i^0
=\phi(P_i)\phi\left( C^*(\mathbf{T})\right)\phi(P_i)= \phi\left(P_iC^*(\mathbf{T})P_i\right) = \phi(C_{\theta_i}).$$
Further, for each $1\leq i \leq k$, there exists a unique  $1 \leq j \leq k$ such that $P_i^0=P_j$, as $U$ and $U^0$ commute, therefore, we get
$$C^0_{\theta_i}= P_i^0C^*(\mathbf{T}^{\prime})P_i^0= P_jC^*(\mathbf{T}^{\prime})P_j= C^{\prime}_{\theta_j}.$$
This proves that $\phi(C_{\theta_i}) = C^{\prime}_{\theta_j}$, and hence $C_{\theta_i} \cong C_{\theta_j}^{\prime}$. By Lemma \ref{plusminus2}, we get $\theta_j=\pm \theta_i$, proving the claim. 
\qed 
\brmrk
For a unitary operator \( U \) with finite spectrum, Theorems \ref{U-twisted} and \ref{U-twisted 1}, along with Lemma \ref{plusminus} and Remark \ref{plusminus1}, provide a complete classification of the \( C^* \)-algebras generated by a pair of \( U \)-doubly twisted isometries with finite-dimensional wandering spaces.
\ermrk
	\noindent\begin{footnotesize}\textbf{Acknowledgement}: 
Bipul Saurabh gratefully acknowledges support from the   NBHM grant 02011/19/2021/NBHM(R.P)/R\&D II/8758.
\end{footnotesize}

		\bigskip
	\noindent{\sc Shreema  Subhash Bhatt} (\texttt{shreemab@iitgn.ac.in},  \texttt{shreemabhatt3@gmail.com})\\
	{\footnotesize Department of Mathematics,\\ Indian Institute of Technology, Gandhinagar,\\  Palaj, Gandhinagar 382055, India}
	\bigskip
	
	\noindent{\sc Surajit Biswas}  (\texttt{surajitb@iiitd.ac.in},
	\texttt{241surajit@gmail.com})\\
	{\footnotesize Department of Mathematics,\\ Indraprastha Institute of Information Technology Delhi,\\ Okhla Industrial Estate, Phase III, New Delhi, Delhi 110020, India}\\

	\noindent{\sc Bipul Saurabh} (\texttt{bipul.saurabh@iitgn.ac.in},  \texttt{saurabhbipul2@gmail.com})\\
	{\footnotesize Department of Mathematics,\\ Indian Institute of Technology, Gandhinagar,\\  Palaj, Gandhinagar 382055, India}


\begin{thebibliography}{10}
	
	\bibitem{Arv-1976aa} 
	Arveson W,
	\newblock {An invitation to $C^*$-algebras}, Grad. Texts in Math., No. 39,
	\newblock Springer-Verlag, New York-Heidelberg, 1976. x+106 pp.
	
	\bibitem{BhaSau-2023aa}
	Bhatt S S and Saurabh B,
	\newblock{$K$-stability of $C^*$-algebras  generated by isometries and unitaries with twisted commutation relations},
	\newblock arXiv:2312.06189, 2023.
	
	\bibitem{Dad-2009aa}
	Dadarlat M,
	\newblock {Continuous fields of $C^*$-algebras over finite dimensional spaces},
	\newblock {\em Adv. Math.} \textbf{222} (2009), no. 5, 1850-1881.
	
	\bibitem{DejPin-2016aa}
	de Jeu M and Pinto P R,
	\newblock { The structure of doubly non-commuting isometries},
	\newblock \emph {Adv. Math.} \textbf{368} (2020), 107149, 35 pp.
	
	\bibitem{Ell-1995aa}
	Elliott 	G A,
	\newblock The classification problem for amenable $C^*$-algebras, 
	\newblock {\em Proceedings of the International Congress of Mathematicians (Zürich, 1994)},
	922-932, 
	1995.
	
	\bibitem{GioKerPhiTom-2017aa}
	Giordano T, Kerr D, Phillips N C and Toms A,
	\newblock{Crossed Products of $C^{*}$ algebras, Topological Dynamics and Classification}
	\newblock  Advanced Courses in Mathematics CRM Barcelona, Birkhauser, 2017.
	
	\bibitem{Kab-2001aa}
	Kabluchko Z A,
	\newblock{On the extensions of higher dimensional non-commutative tori}, 
	\newblock\emph{Methods of Functional Analysis and Topology} \textbf{7} (2001) no.1, pp 28-33.
	
	\bibitem{Kas-1988aa} 
	Kasparov G G, 
	\newblock{Equivariant KK-theory and the Novikov conjecture}, 
	\newblock\emph{Invent. Math.} \textbf{91} (1988) 147–201.
	
	\bibitem{Lin-2009aa}
	Lin H,
	\newblock{Unitary equivalences for essential extensions of  $C^*$-algebras}
	\newblock {\em Proc. Amer. Math. Soc.} 137 (2009), no. 10, 3407-3420.
	
	\bibitem{Phi-2000aa}
	Phillips N C,
	\newblock {A classification theorem for nuclear purely infinite simple $C^*$-algebras},
	\newblock \emph{Doc. Math.} \textbf{5} (2000) 49-114.
	
	
	\bibitem{PimVoi-1980aa}
	Pimsner M V and Voiculescu D V,
	\newblock {Exact sequences for $K$-groups and Ext-groups of certain cross-product $C^*$-algebras},
	\newblock \emph{J. Operator Theory} \textbf{4} (1980) 93-118.
	
	\bibitem{Pro-2000aa}
	Proskurin D,
	\newblock{Stability of a special class of $q_{ij}$-CCR and extensions of higher-dimensional noncommutative tori},
	\newblock \emph{Lett. Math. Phys.} \textbf{52} (2000), no. 2, 165 - 175.
	
	\bibitem{RakSarSur-2022aa}
	Rakshit N, Sarkar J and Suryawanshi M,
	\newblock {Orthogonal decompositions and twisted isometries}, 
	\newblock \emph{Internat. J. Math.} \textbf{33} (2022), Paper No. 2250058, 28 pp. 
	
	\bibitem{RakSarSur-2022ab}
	Rakshit N, Sarkar J and Suryawanshi M,
	\newblock {Orthogonal decompositions and twisted isometries II}, 
	\newblock \emph{J. Operator. Theory} \textbf{1} (2024), 257-281.
	
	\bibitem{Ell-1995aa}
	George A. Elliott,
	\newblock The classification problem for amenable $C^*$-algebras, 
	\newblock {\em Proceedings of the International Congress of Mathematicians (Zürich, 1994)},
	922-932, 
		1995.
	
	\bibitem{Rie-1990aa}
	Rieffel M A,
	\newblock{Noncommutative tori - a case study of noncommutative differentiable manifolds}, Geometric and topological invariants of elliptic operators (Brunswick, ME, 1988), 191 - 211,
	\newblock \emph{Contemp. Math.}, \textbf{105}, American Mathematical Society,  1990.
	
	\bibitem{Sla-1972aa}
	Slawny J,
    \newblock {On factor representations and the $C^*$-algebra of canonical commutation relations},
    \newblock\emph{Comm. Math. Phys.} \textbf{24} (1972), 151-170.
	
	\bibitem{Web-2013aa}
	Weber M,
	\newblock {On $C^*$-algebras generated by isometries with twisted commutation relations},
	\newblock \emph {J. Funct. Anal.} \textbf{264} (2013) 1975-2004.
	
	\end{thebibliography}
\end{document}